\documentclass[10pt]{article}

\usepackage{amssymb,amsmath,amsthm}
\usepackage{booktabs}
\usepackage{amsfonts}
\usepackage{amsthm}
\usepackage{graphicx}
\usepackage{enumerate}
\usepackage[usenames]{color}
\usepackage{url}
\usepackage{subfig}
\usepackage{algorithm,algorithmic} 
\usepackage{dsfont,verbatim}
\usepackage{epsfig}
\usepackage[margin=1.2in]{geometry}
\usepackage{soul}
\usepackage{ulem}
\usepackage[thicklines]{cancel}
\graphicspath{{./pics/}}
\allowdisplaybreaks

\newtheorem{theorem}{Theorem}[section]
\newtheorem{lemma}{Lemma}[section]

\newtheorem{example}{Example}[section]

\newtheorem{remark}{Remark}[section]

\numberwithin{equation}{section}

\title{\bf Optimal $L^2$ error estimates of mass- and energy-
conserved FE schemes for a nonlinear Schr\"{o}dinger--type system}

\author{Zhuoyue Zhang\thanks{Department of Mathematics, School of Sciences, Hangzhou Dianzi University, Hangzhou, Zhejiang, P. R. China. Email address:
\texttt{zzhuoyue@hdu.edu.cn}}
\and Wentao Cai\thanks{Corresponding author; Department of Mathematics, School of Sciences, Hangzhou Dianzi University, Hangzhou, Zhejiang, P. R. China. Email address:
\texttt{femwentao@hdu.edu.cn}}
}
\date{}
\begin{document}
\maketitle

\begin{abstract}
In this paper, we present an implicit Crank--Nicolson finite element (FE) scheme for solving a nonlinear Schr\"{o}dinger--type system, which includes Schr\"{o}dinger--Helmholz system and Schr\"{o}dinger--Poisson system. In our numerical scheme, we employ an implicit Crank--Nicolson method for time discretization and a conforming FE method for spatial discretization. The proposed method is proved to be well-posedness and ensures mass and energy conservation at the discrete level. Furthermore, we prove optimal $L^2$ error estimates for the fully discrete solutions. Finally, some numerical examples are provided to verify the convergence rate and conservation properties.\\[5pt]
\end{abstract}
\pagestyle{myheadings}
\thispagestyle{plain}

\section{Introduction}\label{sec:intro}
We consider the following Schr\"{o}dinger--type system:
\begin{align}                                                     &
{\bf i}u_t+\Delta u-\phi u={0},\label{i1}\\
&
\alpha\phi-\beta^2\Delta\phi=|u|^2,\label{i2}
\end{align}
for $t\in(0,T]$, in a bounded convex polygonal or polyhedral domain $\Omega\subset\mathbb{R}^d\, (d=2,3)$. In this model, ${\bf i}=\sqrt{-1}$ is the imaginary unit, $\alpha$, $\beta$ are non-negative real constants, $u({\bf x},t)$ is a complex-valued function which denotes the single particle wave function and $\phi({\bf x},t)$ is a real-valued function which denotes the potential. Here, the system \eqref{i1}-\eqref{i2} is subject to the following initial and boundary conditions:
\begin{align}\label{i3}
\begin{aligned}
&u({\bf x},0)={u_0}(\bf x),
&&\phi ({\bf x},0)=\phi_0(\bf x),
&&\textrm{in }\Omega,\\
&u({\bf x},t)=0,
&&\phi({\bf x},t)=0,
&&\textrm{on }{\partial\Omega\times(0,T]}.
\end{aligned}
\end{align}

For different parameter $\alpha$, $\beta$, the system \eqref{i1}-\eqref{i3} represents the different models:
\vspace{-5.8pt}

\begin{itemize}
\setlength{\itemsep}{0pt}
\setlength{\parsep}{0pt}
\item[\textbullet]{Schr\"{o}dinger--Helmholz model (as $\alpha \neq 0$ and $\beta\neq 0$)}
\item[\textbullet]{Schr\"{o}dinger--Poisson model (as $\alpha = 0$ and $\beta=1$)}
\end{itemize}
\vspace{-5.8pt}

The nonlinear Schr\"{o}dinger--type equations can be used to describe a variety of physical phenomena in optics, quantum mechanics, and plasma physics, as referenced in \cite{wuli1,wuli2}. The mathematical theory of the system \eqref{i1}-\eqref{i3} has been studied in \cite{Cao,analysis1, analysis2} to establish the well-posedness of strong solutions. It is evident that the Schr\"{o}dinger--type system \eqref{i1}-\eqref{i3} conserves the total mass:
\begin{align}
{\cal M}(t):={\int_\Omega|u(t)|^2}d{\bf x} ={\cal M}(0),\quad t>0,\nonumber
\end{align}
and the total energy
\begin{align}
{\cal E}(t):={\int_\Omega(|\nabla u(t)|^2+\frac{\alpha}{2}|\phi(t)|^2+\frac{\beta^2}{2}|\nabla \phi(t)|^2})d{\bf x}={\cal E}(0),\quad t>0.\nonumber
\end{align}

As pointed out in article \cite{49}, non-conservative numerical schemes may lead to numerical solutions blow up. Therefore, it is desirable to design numerical schemes that ensure mass and energy conservation at discrete level. 
In recent years, numerical methods and analysis for structure-preserving schemes for the Schr\"{o}dinger--type equations have been investigated extensively, e.g., see \cite{fdm1,fdm2,fdm3,fdmsun,fdm4,fdm5}  for finite difference methods, \cite{femD, DG1, DG2,fem4,fem2, DG3,fem5,fem3} for finite element methods, \cite{spectral1,spectral2,spectral3} for spectral or pseudo-spectral methods, and \cite{other1,otherbai,othercao,other2,otherli,other3} for others.
For linearized numerical schemes, Chang et al. \cite{1} proposed several finite difference schemes for solving the generalized nonlinear Schr\"{o}dinger (GNLS) equation. Compared to Hopscotch-type schemes, split step Fourier scheme and pseudospectral scheme, the authors demonstrated that Crank--Nicolson method was more efficient and robust. In \cite{Akrivis}, Akrivis et al. presented a linearized Crank--Nicolson FE scheme for solving GNLS equation. By the energy method, the authors obtained the optimal $L^2$ error estimate under the time step condition $\tau = \mathcal{O}(h^\frac{d}{4})$ $(d=2,3)$, where $\tau$ is the time step size, and $h$ is the spatial step size. Wang \cite{3} investigated the linearized Crank--Nicolson FE scheme for GNLS equation. Inspired by the work \cite{Li}, an error splitting method was employed to analyze the $L^2$ error of FE solutions. By introducing a time-discrete scheme, the authors established uniform boundedness and error estimates of temporal semi-discrete solutions. By the above results and inverse inequality, the uniform boundedness of fully-discrete solutions was demonstrated without grid-ratio restriction conditions. Based on these results, Wang derived unconditionally optimal $L^2$ error estimates of the Crank--Nicolson FE scheme for the GNLS equation. Recently, Wang \cite{4} studied a Crank--Nicolson FE scheme for Schr\"{o}dinger--Helmholtz equations. Utilizing the error splitting technique, optimal $L^2$ error estimates were obtained without any restrictive conditions on time-step size.

Compared to linearized numerical schemes, nonlinear numerical schemes exhibit better stability and preserve more physical structures. In \cite{Akrivis}, Akrivis et al. proposed several implicit Crank--Nicolson FE schemes for the GNLS equation. The optimal $L^2$ error estimates were obtained under certain time-step restrictive conditions. Feng et al. \cite{6} focused on the high-order SAV-Gauss collocation FE schemes for the GNLS equation, which used a Gauss collocation method for temporal discretization and the conforming FE method for spatial discretization. This scheme was shown to ensure mass and energy conservation. The authors also presented the existence, uniqueness and optimal $L^\infty(0,T;H^1)$ error estimates for fully-discrete collocation FE solutions. Henning and Peterseim \cite{8} studied the nonlinear implicit Crank--Nicolson FE method for the GNLS equation. By error splitting technique, the authors obtained the optimal $L^\infty(0,T;L^2)$ error estimate without any restrictive conditions on time-step size. In \cite{7}, a class of nonlinear discontinuous Galerkin (DG) schemes was introduced for the Schr\"{o}dinger--Poisson equations, which consists of the high-order conservative DG schemes and the second-order time discrete schemes. These schemes were proven to conserve mass and energy. Moreover, optimal $L^2$ error estimates were provided for the spatial semi-discrete DG scheme.

In this paper, we propose an implicit Crank--Nicolson FE scheme to solve a nonlinear Schr\"{o}dinger--type system, which includes Schr\"{o}dinger--Helmholz model and Schr\"{o}dinger--Poisson model. In our scheme, the system \eqref{i1}-\eqref{i3} is discretized by an implicit Crank--Nicolson scheme in the time direction and a conforming FE scheme in the spatial direction. The advantage of our numerical scheme is mass and energy conservation, ensuring that the scheme will not lead to instability. Later, we prove the existence and uniqueness of the fully-discrete numerical solutions based on Schaefer's fixed point theorem. Finally, we demonstrate the optimal $L^2$ error estimates of fully discrete numerical solutions for  Schr\"{o}dinger--Helmholz system and Schr\"{o}dinger--Poisson system.

The rest of this paper is organized as follows. In Sect.2, we introduce notations, our numerical scheme and present the main theorem. In Sect.3, we demonstrate the discrete mass and energy conservation properties of our numerical scheme. In Sect.4, we present the well-posedness of the numerical scheme by Schaefer's fixed point theorem. In Sect.5, we prove the optimal $L^2$ error estimates of fully discrete solutions. Finally, some numerical examples in Sect.6 are presented to verify the convergence rate and conservation properties.

\section{Preliminaries and the main results}

Let $\Omega$ be an open, bounded convex polygonal domain in $\mathbb{R}^2$ (or polyhedral domain in $\mathbb{R}^3$) with a smooth boundary $\partial\Omega$.
For any integer $k\geq 0$ and $1\le p\le+\infty$, let ${W^{k,p}}(\Omega)$  be the standard Sobolev space with the abbreviation $H^k(\Omega)={W^{k,2}}(\Omega)$ for simplicity. The norms of $W^{k,p}(\Omega)$ and $H^k(\Omega)$ are denoted by $\|\cdot\|_{W^{k,p}}$ and $\|\cdot\|_{H^k}$, respectively.
We denote by $\mathcal{T}_h$ a quasi-uniform partition of $\Omega$ into triangles $\{{\mathcal K}_j\}_{j=1}^M$ in $\mathbb{R}^2$ or tetrahedrons in $\mathbb{R}^3$, and let $h=\mathop{\max}\limits_{1\le j\le M}\{diam{\mathcal K}_j\}$ denote the mesh size. For a partition ${\mathcal T}_h$, we define standard FE space as
\begin{align}\nonumber
V_h=\{v_h\in {\mathcal C}(\overline\Omega): v_h|_{\mathcal{K}_j}\in P_{r}(\mathcal{K}_j)\textrm{ and }v_h=0\textrm{ on }\partial\Omega,\, \forall\,{\mathcal{K}_j}\in{\mathcal{T}_h}\},
\end{align}
where ${P_r}({\mathcal K}_j)$ is the space of continuous piecewise polynomials of degree $r\,(r\ge 1)$ on ${\mathcal K}_j$.

For any two complex functions $u,v\in{L^2}(\Omega)$, the ${L^2}(\Omega)$ inner product is defined by
\begin{align}
   (u,v) ={\int_\Omega{u({\bf x})v^*({\bf x})}}d{\bf x}, \nonumber
\end{align}
in which $v^*$ denotes the conjugate of $v$.

Let $\{t_n|t_n=n\tau;0\le n\le N\}$ be a uniform partition of the  time interval $[0,T]$ with time step size $\tau=T/N$, and $(u^n,\phi^n)=(u(t_n),\phi(t_n)) $. For any sequence $\{\omega ^n\}_{n=1}^N$,
we define 
\begin{align}
D_\tau\omega^n=\frac{\omega^n-\omega^{n-1}}{\tau}, \quad
\overline\omega^{n-\frac{1}{2}}=\frac{\omega^n+\omega^{n-1}}{2},\quad n=1,2,\dots,N.\nonumber
\end{align}

With the above notations, an implicit Crank--Nicolson FE method for the Schr\"{o}dinger--type system \eqref{i1}-\eqref{i3} is to find $(u_h^n,\phi_h^n)\in \widetilde{V}_h\times {V_h}$ such that
\begin{align}
&
{\bf i}(D_\tau u_h^n,v_h)-(\nabla\overline u_h^{n-\frac{1}{2}},\nabla v_h)-(\overline\phi _h^{n-\frac{1}{2}}\overline u_h^{n-\frac{1}{2}},v_h)=0,\quad v_h\in \widetilde{V}_h,
\label{CN1}\\
&
\alpha(\phi_h^n,w_h)+\beta^2(\nabla\phi_h^n,\nabla w_h)=(|u_h^n|^2,w_h),\quad w_h\in V_h,
\label{CN2}
\end{align}
where $\widetilde{V}_h:=V_h\oplus{\bf i}V_h$ and $n=1,2,\dots,N$. At the initial time step, we choose $u_h^0=R_h u_0$, $\phi_h^0=R_h \phi_0$. Here, $R_h$ denotes the Ritz projection.

In the remaining parts of this paper, we assume that the solution to the initial and boundary value problem \eqref{i1}-\eqref{i3} exists and satisfies
\begin{align}
\|u_t\|_{L^\infty([0,T];H^{r+1})} 
&+\|u_{tt}\|_{L^\infty([0,T];H^2)}
+\|u_{ttt}\|_{L^\infty([0,T];L^2)}\nonumber\\ 
&+\|\phi\|_{L^\infty([0,T];H^{r+1})} 
+\|\phi_{tt}\|_{L^\infty([0,T];L^2)} \le M_0, \label{solution}
\end{align}
where $M_0$ is a positive constant depends only on $\Omega$. 

We present our main results in the following theorem and the proof will be given in section 5.

\begin{theorem}\label{main}
Suppose that the system \eqref{i1}-\eqref{i3} has a unique solution $(u,\phi)$ satisfying the regularity condition \eqref{solution}. 
Then there exists a positive constant $\tau_0$, such that when $\tau\leq\tau_0$, the fully-discrete system \eqref{CN1}-\eqref{CN2} admits a unique FE solution $(u_h^n, \phi_h^n)$ satisfying
\begin{align}
\|u^n-u_h^n\|_{L^2}+\|\phi^n-\phi_h^n\|_{L^2}
\le C_0(\tau^2+h^{r+1}), \quad n=1, 2, ...,N,\nonumber
\end{align}
where $C_0^*$ is a positive constant independent of $n$, $h$ and $\tau$.
\end{theorem}

Next, we provide the following inequality, which will be used frequently in our proof.
\begin{lemma}\label{Gronwall} (Discrete Gronwall's inequality) {\normalfont\cite{G}}
 Let $\tau,B$ and $a_k,b_k,c_k,\gamma_k$ (for integers $k\ge 0$) be nonnegative numbers such that
\begin{align}
a_{\ell}+\tau\sum\limits_{k=0}^\ell b_k
\le\tau\sum\limits_{k=0}^\ell\gamma_k a_k
+\tau\sum\limits_{k=0}^\ell c_k+B,\quad
\textrm{for $\ell\ge 0$}.\nonumber
\end{align}
Suppose that $\tau\gamma_k<1$ for all $k$, and set $\sigma_k=(1-\tau\gamma_k)^{-1}$, then 
\begin{align}
a_{\ell}+\tau\sum\limits_{k=0}^\ell b_k
\le\exp\bigg(\tau\sum\limits_{k=0}^\ell\gamma_k\sigma _k\bigg)
\bigg(\tau\sum\limits_{k = 0}^\ell{c_k}+B\bigg),\quad
\textrm{for $\ell\ge 0$}.\nonumber
\end{align}
\end{lemma}

Throughout this paper, we denote by $C$ a generic constant, and $\varepsilon$ a sufficiently small generic constant,  which could be different at different occurrences.

\section{Discrete conservation laws}
In this section, we demonstrate that the proposed scheme \eqref{CN1}-\eqref{CN2} conserves mass and energy. 

\begin{theorem}\label{energy}
The solution of scheme \eqref{CN1}-\eqref{CN2} satisfies 
the discrete mass and energy conservation:  
\begin{align}
{\mathcal M}_h^n={\mathcal M}_h^0,
\quad{\mathcal E}_h^n={\mathcal E}_h^0,\quad n=1,2,...,N, \label{ME1}
\end{align}
where the discrete mass and energy are defined as
\begin{align}\label{ME2}
{\mathcal M}_h^n:=\|u_h^n\|_{L^2}^2,\quad  
{\mathcal E}_h^n:=
\|\nabla u_h^n\|_{L^2}^2
+\frac{\alpha}{2}\|\phi_h^n\|_{L^2}^2
+\frac{\beta^2}{2}\|\nabla\phi_h^n\|_{L^2}^2.
\end{align}
\end{theorem}

\begin{proof} Choosing $v_h=u_h^n+u_h^{n-1}$ in \eqref{CN1} and taking the imaginary parts, we have
\begin{align}
\|u_h^n\|_{L^2}^2=\|u_h^{n-1}\|_{L^2}^2, \nonumber
\end{align}
which implies
\begin{align}
{\mathcal M}_h^n={\mathcal M}_h^0,\quad\enspace\textrm{for}\enspace n=1,2,...,N.\label{me1}
\end{align} 
Putting $v_h=u_h^n-u_h^{n-1}$ in \eqref{CN1} and taking the real parts, one has 
\begin{align}
{\rm Im}(D_\tau u_h^n,u_h^n-u_h^{n-1}) 
-{\rm Re}(\nabla\overline u_h^{n-\frac{1}{2}},\nabla(u_h^n-u_h^{n-1})) 
-{\rm Re}(\overline\phi_h^{n-\frac{1}{2}}\overline u_h^{n-\frac{1}{2}},u_h^n-u_h^{n-1} )=0,\nonumber
\end{align}
which reduces to
\begin{align}\label{ME3}
-(\|\nabla u_h^n\|_{L^2}^2-\|\nabla u_h^{n-1}\|_{L^2}^2)= 
(\overline\phi_h^{n-\frac{1}{2}}u_h^n,u_h^n)
-(\overline\phi_h^{n-\frac{1}{2}}u_h^{n-1},u_h^{n-1}).
\end{align}
From \eqref{CN2}, it is easy to see
\begin{align} 
\alpha(\phi_h^n-\phi_h^{n-1},w_h)+\beta^2(\nabla(\phi_h^n-\phi_h^{n-1}),\nabla w_h)=(|u_h^n|^2-|u_h^{n-1}|^2,w_h).\nonumber
\end{align}
Taking  $w_h=\overline\phi_h^{n-\frac{1}{2}}$ in the above equation, one has
\begin{align}\nonumber
\frac{\alpha}{2}(\|\phi_h^n\|_{L^2}^2-\|\phi_h^{n-1}\|_{L^2}^2)
+\frac{\beta^2}{2}(\|\nabla\phi_h^n\|_{L^2}^2 -\|\nabla\phi_h^{n-1}\|_{L^2}^2)
=(|u_h^n|^2-|u_h^{n-1}|^2 ,\overline\phi_h^{n-\frac{1}{2}}),
\end{align}
which with \eqref{ME3} implies
\begin{align}
\frac{\alpha}{2}(\|\phi_h^n\|_{L^2}^2-\|\phi_h^{n-1}\|_{L^2}^2) 
+\frac{\beta ^2}{2}(\|\nabla\phi_h^n\|_{L^2}^2 
-\|\nabla\phi_h^{n-1}\|_{L^2}^2) 
+(\|\nabla u_h^n\|_{L^2}^2-\|\nabla u_h^{n-1}\|_{L^2}^2) 
= 0.\nonumber
\end{align}
Thus, we have
\begin{align}
{\mathcal E}_h^n={\mathcal E}_h^0,\quad\enspace\textrm{for }n=1,2,...,N.\label{me2}
\end{align} 

Based on \eqref{me1} and \eqref{me2}, we complete the proof of \eqref{ME1}-\eqref{ME2}
\end{proof}

\begin{remark}
The discrete conservation laws \eqref{ME1}-\eqref{ME2} yields naturally the following regularity result for the numerical solution:
\begin{align}\label{ME4}
\|u_h^n\|_{{L^2}}  
+\|\nabla u_h^n\|_{L^2}
+\|\phi_h^n\|_{L^2}
+\|\nabla\phi_h^n\|_{L^2}
\le C_1,\quad n=1,2,...,N,
\end{align}
where 
\begin{align}
C_1&=\|u_h^0\|_{L^2}
+ \sqrt{\|\nabla u_h^0\|_{L^2}^2+\frac{\alpha}{2}\|\phi_h^0\|_{L^2}^2+\frac{\beta^2}{2}\| \nabla\phi_h^0\|_{L^2}^2 }
+ \sqrt{\frac{2}{\alpha}\|\nabla u_h^0\|_{L^2}^2+\|\phi_h^0\|_{L^2}^2+\frac{\beta^2}{\alpha}\| \nabla\phi_h^0\|_{L^2}^2 }\nonumber\\
&\quad+ \sqrt{ \frac{2}{\beta^2}\|\nabla u_h^0\|_{L^2}^2+\frac{\alpha}{\beta^2}\|\phi_h^0\|_{L^2}^2+\| \nabla\phi_h^0\|_{L^2}^2}.\nonumber
\end{align}
\end{remark}

This result will be used frequently in our analysis process.

\section{Existence and uniqueness of numerical solution}

The existence and uniqueness of the solution of the scheme \eqref{CN1}-\eqref{CN2} will be established in Theorem \ref{TH4-1} and \ref{TH4-2}. In order to get these properties, we begin by introducing the following Schaefer's fixed point theorem.

\begin{lemma}\label{fixed}(Schaefer's Fixed Point Theorem) {\normalfont\cite{Evans}}
Let $B$ be a Banach space, and assume that ${\mathcal G}: B\rightarrow B$ is a continuous and compact mapping. If the set
\begin{align}
\varTheta:=\{(u,\phi)\in B:\exists\,\theta\in[0,1]\ \textrm{such that}\  (u,\phi)=\theta\,
{\mathcal G}(u,\phi)\}\nonumber
\end{align}
is bounded in $B$, then ${\mathcal G}$ has at least one fixed point.
\end{lemma}

\begin{theorem}\label{TH4-1}
For any given $\tau>0$ and $h>0$, there exists a solution $(u_h^n,\phi_h^n)\ (n=1,2,...,N)$ to the scheme \eqref{CN1}-\eqref{CN2}.
\end{theorem}

\begin{proof}  
Firstly, we define map ${\mathcal G}:\widetilde{V}_h\times V_h\rightarrow \widetilde{V}_h\times V_h$ by 
\begin{align}
{\mathcal G}(u_h^n,\phi_h^n)=(u_{h,*}^n,\phi_{h,*}^n),
\quad{\forall\,(u_h^n,\phi_h^n)\in \widetilde{V}_h\times V_h}, \nonumber
\end{align}
where $(u_{h,*}^n,\phi_{h,*}^n)\in \widetilde{V}_h\times V_h$ satisfies the following equations: For given $u_h^{n-1}$, $\phi_h^{n-1}$, $u_h^n$ and $\phi_h^n$,
\begin{align}
&
{\bf i}\,\Big(\frac{u_{h,*}^n-u_h^{n-1}}{\tau},v_h\Big)
-\frac{1}{2}(\nabla u_{h,*}^n,\nabla v_h)
=\frac{1}{2}(\nabla u_h^{n-1},\nabla v_h)
+(\overline\phi_h^{n-\frac{1}{2}}\overline u_h^{n-\frac{1}{2}},v_h),\quad\forall\, v_h\in\widetilde{V}_h,  \label{*1}\\
&
\alpha(\phi_{h,*}^n,w_h)+\beta^2(\nabla\phi_{h,*}^n,\nabla w_h)
=(|u_h^n|^2,w_h),\quad\forall\, w_h\in V_h. \label{*2} 
\end{align}

Next, we will prove the map ${\mathcal G}$ satisfies the three conditions of Lemma \ref{fixed}. Thus, the map ${\mathcal G}$ has a fixed point which is a solution of the scheme \eqref{CN1}-\eqref{CN2}.\\[0.1pt]

{\bf Step 1. Well-defined map}\\ \vspace{-0.3cm}

Given $u_h^n=0$, $u_h^{n-1}=0$, $\phi_h^n=0$, $\phi_h^{n-1}=0$ in \eqref{*1}-\eqref{*2}, we have
\begin{align}
&
\frac{\bf i}{\tau}(u_{h,*}^n,v_h)-\frac{1}{2}(\nabla u_{h,*}^n,\nabla v_h)=0, \qquad\forall\, v_h\in \widetilde{V}_h,\nonumber\\
&
\alpha(\phi_{h,*}^n,w_h)+\beta^2(\nabla\phi_{h,*}^n,\nabla w_h)=0, \,\quad\forall\, w_h\in V_h.\nonumber
\end{align}
Putting $(v_h,w_h)=(u_{h,*}^n,\phi_{h,*}^n)$ in the above equations, we obtain
\begin{align}
\|u_{h,*}^n\|_{L^2}^2=\|\nabla u_{h,*}^n\|_{L^2}^2
=\|\phi_{h,*}^n\|_{L^2}^2=\|\nabla\phi_{h,*}^n\|_{L^2}^2=0.\nonumber
\end{align}
Therefore, the map ${\mathcal G}$ is well-defined.\\[0.1pt]

{\bf Step 2. Continuous and compact}\\ \vspace{-0.3cm}

In order to prove the continuity of the map ${\mathcal G}$, we let 
\begin{align}
{\mathcal G}(\widehat u_h^n,\widehat\phi_h^n)=(\widehat u_{h,*}^n,\widehat\phi_{h,*}^n) 
\quad\textrm{and}\quad 
{\mathcal G}(\mathring u_h^n,\mathring\phi_h^n)=(\mathring u_{h,*}^n,\mathring\phi_{h,*}^n),\nonumber
\end{align}
where $(\widehat u_h^n,\widehat\phi_h^n)\in\widetilde{V}_h\times V_h$ and $(\mathring u_h^n,\mathring\phi_h^n)\in\widetilde{V}_h\times V_h$.\\\vspace{-0.3cm}

Next, we prove that when 
 $\widehat u_h^n\rightarrow \mathring u_h^n$ and $\widehat\phi_h^n\rightarrow\mathring\phi_h^n$,
the following results hold:
\begin{align}
\widehat u_{h,*}^n\rightarrow\mathring u_{h,*}^n\quad\textrm{and} \quad\widehat\phi_{h,*}^n\rightarrow\mathring\phi_{h,*}^n.\label{convergence result}
\end{align}

Since all norms are equivalent in the finite dimensional spaces $\widetilde{V}_h$ and $V_h$, we want to prove \eqref{convergence result} holds, as long as 
\begin{align}
\|\widehat u_{h,*}^n-\mathring u_{h,*}^n\|_{H^1}\rightarrow 0
\quad\textrm{and}\quad\|\widehat\phi_{h,*}^n-\mathring\phi_{h,*}^n\|_{H^1} \rightarrow 0.\nonumber
\end{align}

Next, we will demonstrate these results. By the definition of the map ${\mathcal G}$, we get
\begin{align}
&
\frac{\bf i}{\tau}(\widehat u_{h,*}^n-\mathring u_{h,*}^n,v_h)
-\frac{1}{2}(\nabla(\widehat u_{h,*}^n-\mathring u_{h,*}^n),\nabla v_h)\nonumber\\
&\quad=
\frac{1}{4}((\widehat\phi_h^n+\phi_h^{n-1})
(\widehat u_h^n+u_h^{n-1})
-(\mathring\phi_{h}^n+\phi_h^{n-1})
(\mathring u_h^n+u_h^{n-1}),v_h),\label{*5}\\
&
\alpha (\widehat\phi_{h,*}^n-\mathring \phi_{h,*}^n,w_h)
+\beta^2(\nabla(\widehat\phi_{h,*}^n-\mathring\phi_{h,*}^n),\nabla w_h)
=(|\widehat u_h^n|^2-|\mathring u_{h}^n|^2,w_h)\label{*6} 
\end{align}
for any $v_h\in\widetilde{V}_h,\,w_h\in V_h$.\\
Substituting $v_h=\widehat u_{h,*}^n-\mathring u_{h,*}^n$ into \eqref{*5}, we obtain
\begin{align}
&
\frac{\bf i}{\tau}\|\widehat u_{h,*}^n-\mathring u_{h,*}^n\|_{L^2}^2
-\frac{1}{2}\|\nabla(\widehat u_{h,*}^n-\mathring u_{h,*}^n)\|_{L^2}^2\nonumber\\
&=
\frac{1}{4}((\widehat\phi_h^n+\phi_h^{n-1}) 
(\widehat u_h^n-\mathring u_h^n),\widehat u_{h,*}^n-\mathring u_{h,*}^n)
+\frac{1}{4}((\widehat\phi_h^n-\mathring\phi_h^n)
(\mathring u_h^n+u_h^{n-1}),\widehat u_{h,*}^n-\mathring u_{h,*}^n).\label{*6.5} 
\end{align}

Taking the real parts and imaginary parts of the equation \eqref{*6.5} respectively, and then adding them together, we get
\begin{align}
&
\frac{1}{\tau}\|\widehat u_{h,*}^n-\mathring u_{h,*}^n\|_{L^2}^2+
\frac{1}{2}\|\nabla(\widehat u_{h,*}^n-\mathring u_{h,*}^n)\|_{L^2}^2
\nonumber\\
&\le 
C(\|\widehat\phi_h^n+\phi_h^{n-1}\|_{L^2}
\|\widehat u_h^n-\mathring u_h^n\|_{L^4}
+\|\widehat\phi_h^n-\mathring \phi_h^n\|_{L^4}
\|\mathring u_h^n+u_h^{n-1}\|_{L^2})
\|\widehat u_{h,*}^n-\mathring u_{h,*}^n\|_{L^4}\nonumber\\
&\le
\varepsilon\|\nabla(\widehat u_{h,*}^n-\mathring u_{h,*}^n)\|_{L^2}^2
+C\|\widehat\phi_h^n+\phi_h^{n-1}\|_{L^2}^2 
\|\widehat u_h^n-\mathring u_h^n\|_{L^4}^2
+C\|\widehat\phi_h^n-\mathring\phi_h^n\|_{L^4}^2
\|\mathring u_h^n+u_h^{n-1}\|_{L^2}^2.\nonumber
\end{align}

Since $u_h^{n-1}$, $\mathring u_h^n$ $\in \widetilde{V}_h\hookrightarrow L^2(\Omega)$ and 
$\phi_h^{n-1}$ , $\widehat\phi_h^n$ $\in V_h\hookrightarrow L^2(\Omega)$, then $\|\widehat\phi_h^n+\phi_h^{n-1}\|_{L^2}$ and $\|\mathring u_h^n+u_h^{n-1}\|_{L^2}$ are boundedness, which with the above inequality lead to
\begin{align}
\frac{1}{\tau}\|\widehat u_{h,*}^n-\mathring u_{h,*}^n\|_{L^2}^2+\frac{1}{2}\|\nabla(\widehat u_{h,*}^n-\mathring u_{h,*}^n)\|_{L^2}^2
\le C\|\widehat u_h^n-\mathring u_h^n\|_{L^4}^2
+C\|\widehat\phi_h^n-\mathring\phi_h^n\|_{L^4}^2.\nonumber
\end{align}
Thus, when $\widehat u_h^n\rightarrow\mathring u_h^n$ and $\widehat\phi_h^n\rightarrow\mathring\phi_h^n $, we get
\begin{align}
\|\widehat u_{h,*}^n-\mathring u_{h,*}^n\|_{H^1}
\le\|\widehat u_{h,*}^n-\mathring u_{h,*}^n\|_{L^2}
+\|\nabla(\widehat u_{h,*}^n-\mathring u_{h,*}^n)\|_{L^2}\rightarrow 0.\label{continous1}
\end{align}

Substituting $w_h=\widehat\phi_{h,*}^n-\mathring\phi_{h,*}^n$ into \eqref{*6} yields
\begin{align}
&
\alpha\|\widehat\phi_{h,*}^n-\mathring\phi_{h,*}^n\|_{L^2}^2
+\beta^2\|\nabla(\widehat\phi_{h,*}^n-\mathring\phi_{h,*}^n)\|_{L^2}^2\nonumber\\
&
=(|\widehat u_h^n|^2-|\mathring u_h^n|^2,\widehat\phi_{h,*}^n-\mathring\phi_{h,*}^n)\nonumber\\
&
=((|\widehat u_h^n|+|\mathring u_h^n|)(|\widehat u_{h}^n|-|\mathring u_h^n|),\widehat\phi_{h,*}^n-\mathring\phi_{h,*}^n)\nonumber\\
&
\le
(\|\widehat u_h^n\|_{L^4}+\|\mathring u_h^n\|_{L^4})\|\widehat u_{h}^n-\mathring u_h^n\|_{L^4}\|\widehat\phi_{h,*}^n-\mathring\phi_{h,*}^n\|_{L^2}\nonumber\\
&
\le C\|\widehat u_{h}^n-\mathring u_h^n\|_{L^4}^2
+\varepsilon\alpha
\|\widehat\phi_{h,*}^n-\mathring\phi_{h,*}^n\|_{L^2}^2. \quad\mbox{(use $\widehat u_h^n$ and $\mathring u_h^n\in\widetilde{V}_h\hookrightarrow L^4$)}\nonumber
\end{align}
Thus, choosing $\varepsilon$ to be a sufficiently small constant, we obtain
\begin{align*}
&
\alpha\|\widehat\phi_{h,*}^n-\mathring\phi_{h,*}^n\|_{L^2}^2
+\beta^2\|\nabla(\widehat\phi_{h,*}^n-\mathring\phi_{h,*}^n)\|_{L^2}^2
\le C\|\widehat u_h^n-\mathring u_h^n\|_{L^4}^2\rightarrow 0
\end{align*}
as $\widehat u_h^n\rightarrow\mathring u_h^n$, which leads to 
\begin{align}
\|\widehat\phi_{h,*}^n-\mathring\phi_{h,*}^n\|_{H^1}
\le \|\widehat\phi_{h,*}^n-\mathring\phi_{h,*}^n\|_{L^2}
+\|\nabla(\widehat\phi_{h,*}^n-\mathring\phi_{h,*}^n)\|_{L^2}\rightarrow 0\label{continous2}
\end{align}
as $\widehat u_h^n\rightarrow\mathring u_h^n $.

Combining \eqref{continous1} and \eqref{continous2}, we complete the proof of the continuity of the map ${\mathcal G}$. Additionally, since the FE space $\widetilde{V}_h\times V_h$ is finite-dimensional, it follows that ${\mathcal G}$ is a compact map.
\\[2pt]

 {\bf Step 3. Uniform boundedness of the set $\varTheta$} \\ \vspace{-0.3cm}
 
Next, we prove the uniform boundedness of the set $\varTheta$:
\begin{align}
\|u_{h,*}^n\|_{H^1} +\|\phi_{h,*}^n\|_{H^1} \le M,\quad\forall\,(u_{h,*}^n,\phi_{h,*}^n) \in \varTheta, \nonumber
\end{align}
where $M$ is a positive constant independent of $u_{h,*}^n$ and $\phi_{h,*}^n$.

For any $(u_{h,*}^n,\phi_{h,*}^n)\in\varTheta$, we have
\begin{align}
\theta{\mathcal G}(u_{h,*}^n,\phi_{h,*}^n)=(u_{h,*}^n,\phi_{h,*}^n),\quad\theta\in[0,1],\nonumber
\end{align}
that is 
\begin{align}
&
\frac{\bf i}{\tau}(u_{h,*}^n-\theta u_h^{n-1},v_h)-\frac{1}{2}(\nabla u_{h,*}^n,\nabla v_h)\nonumber\\
&
\qquad=\frac{\theta}{2}(\nabla u_{h}^{n-1},\nabla v_h)+\frac{\theta}{4}((\phi_{h,*}^n+\phi_h^{n-1})(u_{h,*}^n+u_h^{n-1}),v_h),\label{bounded1}\\
&
\alpha(\phi_{h,*}^n,w_h)+\beta^2(\nabla\phi_{h,*}^n,\nabla w_h)=\theta(|u_{h,*}^n|^2,w_h) \label{bounded2}
\end{align}
for any $v_h\in\widetilde{V}_h,\, w_h\in V_h$.


Substituting $w_h=\phi_{h,*}^n+\phi_h^{n-1}$ into \eqref{bounded2}
, we obtain
\begin{align}
\theta(|u_{h,*}^n|^2-|u_h^{n-1}|^2,\phi_{h,*}^n+\phi_h^{n-1} )
&
=-\theta(|u_h^{n-1}|^2,\phi_{h,*}^n+\phi_h^{n-1})+\alpha\|\phi_{h,*}^n\|_{L^2}^2+\beta^2\|\nabla\phi_{h,*}^n\|_{L^2}^2\nonumber\\
&\quad +\alpha\|\phi_{h,*}^n\|_{L^2}\|\phi_h^{n-1}\|_{L^2}+\beta^2\|\nabla\phi_{h,*}^n\|_{L^2}\|\nabla\phi_h^{n-1}\|_{L^2}^2\label{bounded2.5}
\end{align}
Further, taking $v_h=u_{h,*}^n-u_h^{n-1}$ in \eqref{bounded1}, we have
\begin{align}
&
\frac{\bf i}{2}(\|u_{h,*}^n\|_{L^2}^2-\|u_h^{n-1}\|_{L^2}^2+\|u_{h,*}^n-u_h^{n-1}\|_{L^2}^2)\nonumber\\
&
\quad-\frac{\tau}{4}(\|\nabla u_{h,*}^n\|_{L^2}^2-\|\nabla u_h^{n-1}\|_{L^2}^2+\|\nabla(u_{h,*}^n-u_h^{n-1})\|_{L^2}^2)\nonumber\\
&
={\bf i}\theta(u_h^{n-1},u_{h,*}^n-u_h^{n-1})+\frac{\theta\tau}{2}(\nabla u_h^{n-1},\nabla(u_{h,*}^n-u_h^{n-1}))\nonumber\\
&
\quad+\frac{\theta\tau}{4}(\phi_{h,*}^n+\phi_h^{n-1},|u^n_{h,*}|^2-|u_h^{n-1}|^2)\nonumber\\
&
={\bf i}\theta(u_h^{n-1},u_{h,*}^n-u_h^{n-1})+\frac{\theta\tau}{2}(\nabla u_h^{n-1},\nabla(u_{h,*}^n-u_h^{n-1}))\nonumber\\
&\quad
+\frac{\alpha\tau}{4}\|\phi_{h,*}^n\|_{L^2}^2+\frac{\beta^2\tau}{4}\|\nabla\phi_{h,*}^n\|_{L^2}^2-\frac{\theta\tau}{4}(|u_h^{n-1}|^2,\phi_{h,*}^n+\phi_h^{n-1})\nonumber\\
&\quad
+\frac{\alpha\tau}{4}\|\phi_{h,*}^n\|_{L^2}\|\phi_h^{n-1}\|_{L^2}+\frac{\beta^2\tau}{4}\|\nabla\phi_{h,*}^n\|_{L^2}\|\nabla\phi_h^{n-1}\|_{L^2}^2,\label{bounded3}
\end{align}
where we use \eqref{bounded2.5} to get the last equality.

Taking the real parts of \eqref{bounded3}, we get
\begin{align}
&
\frac{\tau}{4}\|\nabla u_{h,*}^n\|_{L^2}^2+\frac{\alpha\tau}{4}\|\phi_{h,*}^n\|_{L^2}^2 +\frac{\beta^2\tau}{4}\|\nabla\phi_{h,*}^n\|_{L^2}^2+\|\nabla(u^n_{h,*}-u^{n-1}_h)\|^2_{L^2}\nonumber \\
&
=\frac{\tau}{4}\|\nabla u^{n-1}_h\|^2_{L^2}
-\theta{\rm Im}(u_h^{n-1},u_{h,*}^n-u_h^{n-1})-\frac{\theta\tau}{2}{\rm Re}(\nabla u_h^{n-1},\nabla(u_{h,*}^n-u_h^{n-1}))\nonumber\\
&\quad
+\frac{\theta\tau}{4}(|u_h^{n-1}|^2,\phi_{h,*}^n+\phi_h^{n-1})-\frac{\alpha\tau}{4}\|\phi_{h,*}^n\|_{L^2}\|\phi_h^{n-1}\|_{L^2}-\frac{\beta^2\tau}{4}\|\nabla\phi_{h,*}^n\|_{L^2}\|\nabla\phi_h^{n-1}\|_{L^2}
\nonumber\\
&\le
(\frac{\tau}{4}+\frac{\theta\tau}{2})\|\nabla u^{n-1}_h\|^2_{L^2}+\theta\|u^{n-1}_h\|_{L^2}\|u^n_{h,*}\|_{L^2}+\frac{\theta\tau}{2}\|\nabla u^{n-1}_h\|_{L^2}\|\nabla u^n_{h,*}\|_{L^2}
\nonumber\\
&\quad
+\frac{C\theta\tau}{4}\|u^{n-1}_h\|_{L^2}\|\nabla u^{n-1}_h\|_{L^2}(\|\nabla 
 \phi^n_{h,*}\|_{L^2}+\|\nabla\phi^{n-1}_h\|_{L^2})\nonumber\\
&\quad
+\frac{\alpha\tau}{8}(\|\phi^n_{h,*}\|_{L^2}^2+\|\phi^{n-1}_h\|^2_{L^2})+\frac{\beta^2\tau}{4}(\frac{1}{4}\|\nabla\phi^n_{h,*}\|^2_{L^2}+\|\nabla\phi^{n-1}_h\|^2_{L^2})\nonumber\\
&\le
(\frac{\tau}{4}+\frac{3\theta\tau}{2})\|\nabla u^{n-1}_h\|^2_{L^2}+2\theta\|u^{n-1}_h\|^2_{L^2}+\frac{\theta}{8}\|u^n_{h,*}\|^2_{L^2}+\frac{\theta\tau}{16}\|\nabla u^n_{h,*}\|^2_{L^2}     \nonumber\\
&\quad
+\frac{\theta\tau}{4}\Big(\frac{C}{\beta^2}\|u^{n-1}_h\|^2_{L^2}\|\nabla u^{n-1}_h\|^2_{L^2}+\frac{\beta^2}{4}\Big(\|\nabla \phi^n_{h,*}\|^2_{L^2}+\|\nabla \phi^{n-1}_h\|_{L^2}^2\Big)\Big)    \nonumber\\
&\quad
+\frac{\alpha\tau}{8}\|\phi^n_{h,*}\|^2_{L^2}+\frac{\beta^2\tau}{16}\|\nabla\phi^n_{h,*}\|^2_{L^2}+\frac{\alpha\tau}{8}\|\phi^{n-1}_h\|^2_{L^2}+\frac{\beta^2\tau}{4}\|\nabla\phi^{n-1}_h\|^2_{L^2}\nonumber\\
&\le 
(\frac{\tau}{4}+\frac{3\theta\tau}{2})\|\nabla u^{n-1}_h\|^2_{L^2}+2\theta\|u^{n-1}_h\|^2_{L^2}+\frac{\theta}{8}\|u^n_{h,*}\|^2_{L^2}+\frac{\theta\tau}{16}\|\nabla u^n_{h,*}\|^2_{L^2}+C\tau\|u^{n-1}_h\|^2_{L^2}\|\nabla u^{n-1}_h\|^2_{L^2} \nonumber\\
&\quad+\frac{\alpha\tau}{8}\|\phi^n_{h,*}\|^2_{L^2}+\frac{(1+\theta)\beta^2\tau}{16}\|\nabla\phi^n_{h,*}\|_{L^2}+\frac{\alpha\tau}{8}\|\phi^{n-1}_h\|^2_{L^2}+\frac{(4+\theta)\beta^2\tau}{16}\|\nabla\phi^{n-1}_h\|^2_{L^2}
.\label{bounded3.1}
\end{align}

Moreover, taking the imaginary parts of \eqref{bounded3}, we get
\begin{align}
\frac{1}{2} \|u_{h,*}^n\|_{L^2}^2
&
\le\frac{1}{2}\|u_h^{n-1}\|_{L^2}^2+\left|{\rm Re}\left(\theta(u_h^{n-1},u_{h,*}^n-u_h^{n-1})\right)
+{\rm Im}\Big(\frac{\theta\tau}{2}(\nabla u_h^{n-1},\nabla(u_{h,*}^n-u_h^{n-1})) \Big)\right|\nonumber\\
&\le
(\frac12+\theta)\|u_h^{n-1}\|_{L^2}^2 +\theta\|u_h^{n-1}\|_{L^2}\|u_{h,*}^{n}\|_{L^2} +\frac{\theta\tau}{2}\|\nabla u_h^{n-1}\|_{L^2}\|\nabla u_{h,*}^{n}\|_{L^2}\nonumber\\
&\le
(\frac12+3\theta)\|u_h^{n-1}\|_{L^2}^2+\frac{\theta}{8}\|u_{h,*}^{n}\|_{L^2}^2+\theta\tau\|\nabla u_h^{n-1}\|_{L^2}^2+\frac{\theta\tau}{16}\|\nabla u_{h,*}^{n}\|_{L^2}^2.\label{bounded3.2}
\end{align}
 
Adding \eqref{bounded3.1} and \eqref{bounded3.2}, we have
\begin{align}
&
\frac{1}{2}\|u_{h,*}^n\|_{L^2}^2+\frac{\tau}{4}\|\nabla u_{h,*}^n\|_{L^2}^2+\frac{\alpha\tau}{4}\|\phi_{h,*}^n\|_{L^2}^2+\frac{\beta^2\tau}{4}\|\nabla\phi_{h,*}^n\|_{L^2}^2\nonumber\\
&\le 
\frac{\theta}{4}\|u_{h,*}^n\|_{L^2}^2+\frac{\theta\tau}{8}\|\nabla u_{h,*}^n\|_{L^2}^2+\frac{\alpha\tau}{8}\|\phi^n_{h,*}\|^2_{L^2}+\frac{(1+\theta)\beta^2\tau}{16}\|\nabla\phi^n_{h,*}\|^2_{L^2} +\frac{10\theta+1}{2}\|u_h^{n-1}\|_{L^2}^2\nonumber\\
&
\quad
+\frac{(10\theta+1)\tau}{4}\|\nabla u_h^{n-1}\|_{L^2}^2+C\tau\|u_h^{n-1}\|_{L^2}^2\|\nabla u_h^{n-1}\|_{L^2}^2+\frac{\alpha\tau}{8}\|\phi^{n-1}_h\|^2_{L^2}+\frac{(4+\theta)\beta^2\tau}{16}\|\nabla\phi^{n-1}_h\|^2_{L^2}.\nonumber
\end{align}

Since $u_h^{n-1}\in \widetilde{V}_h\hookrightarrow H^1(\Omega)$ and $\phi_h^{n-1}\in V_h\hookrightarrow H^1(\Omega)$, then $\|u_h^{n-1}\|_{L^2}$, $\|\nabla u_h^{n-1}\|_{L^2}$, $\|\phi_h^{n-1}\|_{L^2}$ and $\|\nabla\phi_h^{n-1}\|_{L^2}$ are boundedness, which with the above inequality lead to 
\begin{align}
\|u_{h,*}^n\|^2_{H^1}
+\|\phi_{h,*}^n\|^2_{H^1}
\le\|u_{h,*}^n\|_{L^2}^2+\|\nabla u_{h,*}^n\|_{L^2}^2+\|\phi_{h,*}^n\|_{L^2}^2+\|\nabla\phi _{h,*}^n\|_{L^2}^2\le M,\nonumber
\end{align}
where $M>0$ is a constant independent of $u_{h,*}^n$ and $\phi_{h,*}^n$. Therefore, we complete the proof of the uniform boundedness of the set $\varTheta$.

Based on the above analysis, we conclude that the map ${\mathcal G}$ satisfies the three conditions in Lemma \ref{fixed}. Thus, the scheme \eqref{CN1}-\eqref{CN2} has a solution, which is a fixed point of the map ${\mathcal G}$.
\end{proof}

Next, we will prove the uniqueness of the solution to the system \eqref{CN1}-\eqref{CN2}.

\begin{theorem}\label{TH4-2}
Under the conditions of Lemma \ref{fixed}, there exists a constant $\tau^*>0$, so that when $\tau\le\tau^*$, the solution $(u_h^n,\phi_h^n)$ $(n=1,2,...,N)$ to the scheme \eqref{CN1}-\eqref{CN2} is unique.
\end{theorem}

\begin{proof}
Let ${(u_{h,1}^n,\phi_{h,1}^n)}$ and ${(u_{h,2}^n,\phi_{h,2}^n)}$ be two FE solutions of the scheme \eqref{CN1}-\eqref{CN2}. we denote
\begin{align}
\widetilde u_h^n=u_{h,1}^n-u_{h,2}^n\quad\textrm{and}
\quad\widetilde\phi_h^n=\phi_{h,1}^n-\phi_{h,2}^n.\nonumber
\end{align}
From \eqref{CN1}-\eqref{CN2}, we have  
\begin{align}
&
\frac{\bf i}{\tau}(\widetilde u_h^n ,v_h) -\frac{1}{2}(\nabla\widetilde u_h^n,\nabla v_h)\nonumber\\
&
=\frac{1}{4}((\phi_{h,1}^n+\phi_h^{n-1})(u_{h,1}^n+u_h^{n-1})-(\phi_{h,2}^n+\phi_h^{n-1})(u_{h,2}^n+u_h^{n-1}),v_h), \quad v_h\in\widetilde{V}_h\label{unique1}\\
&
\alpha(\widetilde\phi_h^n,w_h)+\beta^2(\nabla\widetilde\phi_h^n,\nabla w_h)=(|u_{h,1}^n|^2-|u_{h,2}^n|^2,w_h),\quad w_h\in V_h.
\label{unique2}
\end{align}
Substituting $v_h=\widetilde u_h^n$ into \eqref{unique1}, we obtain
\begin{align}
&
\frac{\bf i}{\tau}\|\widetilde u_h^n\|_{L^2}^2-\frac{1}{2}\|\nabla\widetilde u_h^n\|_{L^2}^2\nonumber \\
&
=\frac{1}{4}((\phi_{h,1}^n+\phi_h^{n-1})(u_{h,1}^n+u_h^{n-1})-(\phi_{h,2}^n+\phi_h^{n-1} )(u_{h,2}^n+u_h^{n-1}),\widetilde u_h^n)\nonumber\\
&
=\frac{1}{4}(\widetilde\phi_h^n u_{h,1}^n,\widetilde u_h^n)+\frac{1}{4}(\widetilde\phi_h^n u_h^{n-1},\widetilde u_h^n)+\frac{1}{4}(\phi_{h,2}^n\widetilde u_{h}^n,\widetilde u_h^n)+\frac{1}{4}(\phi_h^{n-1}\widetilde u_h^n,\widetilde u_h^n).\label{unique3}
\end{align}
By taking the imaginary parts of \eqref{unique3}, we obtain
\begin{align}
\frac{1}{\tau}\|\widetilde u_h^n\|_{L^2}^2 
&=
\frac{1}{4}{\rm Im} 
\left((\widetilde\phi_h^n u_{h,1}^n,\widetilde u_h^n)+(\widetilde\phi_h^n u_h^{n-1},\widetilde u_h^n)+(\phi_{h,2}^n \widetilde u_h^n,\widetilde u_h^n)+(\phi_h^{n-1}\widetilde u_h^n,\widetilde u_h^n)\right)\nonumber\\
&=
\frac{1}{4}{\rm Im}\left((\widetilde\phi_h^n u_{h,1}^n,\widetilde u_h^n)+(\widetilde\phi_h^n u_h^{n-1},\widetilde u_h^n)\right)\nonumber\\
&\le
C\|\widetilde\phi_h^n\|_{L^6}\|u_{h,1}^n\|_{L^3}\|\widetilde u_h^n\|_{L^2}+C\|\widetilde\phi_h^n\|_{L^6}\|u_{h}^{n-1}\|_{L^3}\|\widetilde u_h^n\|_{L^2}\nonumber\\
&\le
C\|\nabla\widetilde\phi_h^n\|_{L^2}\|\nabla u_{h,1}^n\|_{L^2}\|\widetilde u_h^n\|_{L^2}+C\|\nabla \widetilde\phi_h^n\|_{L^2}\|\nabla u_h^{n-1}\|_{L^2}\|\widetilde u_h^n\|_{L^2}\nonumber\\
&\le
C\|\nabla\widetilde\phi_h^n\|_{L^2}^2+C\|\widetilde u_h^n\|_{L^2}^2.
\qquad\qquad \mbox{(use \eqref{ME4})}\label{unique3.5}
\end{align}
Further, taking the real parts of \eqref{unique3}, we get
\begin{align}
\|\nabla\widetilde u_h^n\|_{L^2}^2
&\le
\frac{1}{2}\left|{\rm Re}\left((\widetilde\phi_h^n u_{h,1}^n,\widetilde u_h^n )+(\widetilde\phi_h^n u_h^{n-1},\widetilde u_h^n)+(\phi_{h,2}^n\widetilde u_h^n,\widetilde u_h^n)+(\phi_h^{n-1}\widetilde u_h^n,\widetilde u_h^n)\right)\right|\nonumber\\
&\le
C\|\widetilde\phi_h^n\|_{L^6}\|u_{h,1}^n\|_{L^3}\|\widetilde u_h^n\|_{L^2}+C\|\widetilde\phi_h^n\|_{L^6}\|u_{h}^{n-1}\|_{L^3}\|\widetilde u_h^n\|_{L^2}\nonumber\\ 
&\quad
+C\|\phi_{h,2}^n\|_{L^6}\|\widetilde u_h^n\|_{L^3}\|\widetilde u_h^n\|_{L^2}+C\|\phi_h^{n-1}\|_{L^6}\|\widetilde u _h^n\|_{L^3}\|\widetilde u _h^n\|_{L^2}
\nonumber\\
&\le
C\|\nabla\widetilde\phi_h^n\|_{L^2}\|\nabla u_{h,1}^n\|_{L^2}\|\widetilde u_h^n\|_{L^2}
+C\|\nabla\widetilde\phi_h^n\|_{L^2}\|\nabla u_{h}^{n-1}\|_{L^2}\|\widetilde u_h^n\|_{L^2}
\nonumber\\ 
&\quad
+C\|\nabla\phi_{h,2}^n\|_{L^2}\|\nabla\widetilde u_h^n\|_{L^2}\|\widetilde u_h^n\|_{L^2}
+C\|\nabla\phi_h^{n-1}\|_{L^2}\|\nabla\widetilde u _h^n\|_{L^2}\|\widetilde u _h^n\|_{L^2}
\nonumber\\
&\le
C\|\nabla\widetilde\phi_h^n\|_{L^2}^2 +C\|\widetilde u_h^n\|_{L^2}^2 +\varepsilon\|\nabla\widetilde u_h^n \|_{L^2}^2.
\qquad\qquad\mbox{(use \eqref{ME4})}\nonumber
\end{align}
It follows from the above inequality and \eqref{unique3.5} that
\begin{align}
\|\widetilde u_h^n\|_{L^2}^2+\tau\|\nabla\widetilde u_h^n\|_{L^2}^2
\le C\tau\|\nabla\widetilde\phi_h^n\|_{L^2}^2\label{unique4}
\end{align}
as $\tau\le\tau^*_1$ ($\tau^*_1>0$ is a certain constant).\\
Taking $w_h=\widetilde \phi _h^n$ in \eqref{unique2}, along with \eqref{unique4}, we have
\begin{align}
\alpha\|\widetilde\phi_h^n\|_{L^2}^2+\beta^2\|\nabla\widetilde\phi_h^n\|_{L^2}^2 
&=
(|u_{h,1}^n|^2-|u_{h,2}^n|^2,\widetilde\phi_h^n)=(u_{h,1}^n(\widetilde u_h^n)^*+\widetilde u_h^n (u_{h,2}^n)^*,\widetilde\phi_h^n)\nonumber\\
&\le
\|u_{h,1}^n\|_{L^3}\|\widetilde u_h^n\|_{L^2}\|\widetilde\phi_h^n\|_{L^6}+\|\widetilde u_h^n\|_{L^2}\|u_{h,2}^n\|_{L^3}\|\widetilde\phi_h^n\|_{L^6}\nonumber\\
&\le
C\|\nabla u_{h,1}^n\|_{L^2}\|\widetilde u_h^n\|_{L^2}\|\nabla\widetilde\phi_h^n\|_{L^2} 
+C\|\widetilde u_h^n\|_{L^2}\|\nabla u_{h,2}^n\|_{L^2}\|\nabla\widetilde\phi_h^n\|_{L^2}\nonumber\\
&\le
C\|\widetilde u_h^n\|_{L^2}^2+\varepsilon\|\nabla\widetilde\phi_h^n\|_{L^2}^2
\qquad\qquad\mbox{(use \eqref{ME4})}\nonumber\\
&\le
(C\tau+\varepsilon)\|\nabla\widetilde\phi_h^n\|_{L^2}^2.\nonumber
\end{align}
Combining the above result and \eqref{unique4}, we have
\begin{align}
\|\widetilde u_h^n\|_{L^2}^2+\tau\|\nabla\widetilde u_h^n\|_{L^2}^2+\alpha\|\widetilde\phi_h^n\|_{L^2}^2+\beta^2\|\nabla\widetilde\phi_h^n\|_{L^2}^2
&\le 0\nonumber
\end{align}
as $\tau\le\tau^*_2$ ($\tau^*_2>0$ is a certain constant).

Thus, let $\tau^*=\min\{\tau^*_1,\tau^*_2\}$, we acquire
\begin{align}
\|\widetilde u_h^n\|_{L^2}=\|\nabla\widetilde u_h^n\|_{L^2}   
=\|\widetilde\phi_h^n\|_{L^2}=\|\nabla\widetilde\phi_h^n\|_{L^2}
=0\nonumber
\end{align}
as $\tau\le\tau^*$. The proof of Theorem \ref{TH4-2} is complete.
\end{proof}

\section{Optimal $L^2$ error estimates}
In order to get the optimal $L^2$ error estimates, we introduce the Ritz projection  $R_h:H_0^1(\Omega)\to{V_h}$, defined by 
\begin{align}
(\nabla(v-R_h v),\nabla w_h)=0,\quad \forall\,w_h\in V_h.\nonumber
\end{align}   
By the classical FE theory \cite{ritz1,ritz2}, it holds that 
\begin{align}
&\|v-R_h v\|_{L^2}+h\|\nabla(v-R_h v)\|_{L^2}\le Ch^{r+1}\|v\|_{H^{r+1}},\quad\forall\, v\in H^{r+1}(\Omega)\cap H_0^1(\Omega)\label{ritz1}\\
&\|R_h v\|_{L^\infty}\le C\|v\|_{H^2},  \quad \forall\, v\in H^2(\Omega)\cap H_0^1(\Omega).\label{ritz2}
\end{align}
Let $\Pi_h:C(\Omega)\to V_h$ be the Lagrange interpolation operator. By the classical interpolation theory \cite{interpolation1,ritz2}, it is easy to see that
\begin{align}\label{interpolation}
\|v- \Pi_h v\|_{L^2}+h\|\nabla(v- \Pi_h v)\|_{L^2}
\le Ch^{\ell+1}\|v\|_{H^{\ell+1}},\quad\, \forall\, v\in H^{\ell+1}(\Omega), \,\, 1\le\ell\le r.
\end{align}
 
For the numerical solution $(u_h^n,\phi_h^n)$ of the system \eqref{CN1}-\eqref{CN2}, we denote 
\begin{align*}
e_u^n=R_h u^n-u_h^n \quad\mbox{and}\quad 
e_\phi^n=R_h \phi^n-\phi_h^n.
\end{align*}
In the following subsections, we will give optimal $L^2$ error estimates of the fully-discrete scheme \eqref{CN1}-\eqref{CN2} of Schr\"{o}dinger--Helmholz equations (as  $\alpha \neq 0$ and $\beta\neq 0$) and Schr\"{o}dinger--Poisson equations (as $\alpha = 0$ and $\beta=1$), respectively.

\subsection{Error estimates of the Schr\"{o}dinger--Helmholz equations}
The weak form of system \eqref{i1}-\eqref{i3} can be written as
\begin{align}
&
{\bf i}(D_\tau u^n,v)-(\nabla\overline u^ {n-\frac{1}{2}},\nabla  v)-(\overline\phi^{n-\frac{1}{2}}\overline u^{n-\frac{1}{2}},v)=( R^{n-\frac{1}{2}},v),\label{weak1}\\
&
\alpha(\phi^n,w) 
+\beta^2(\nabla\phi^n,\nabla w)=(|u^n|^2,w)\label{weak2}
\end{align}
for any $v\in H^1_0(\Omega)$, $w\in H^1_0(\Omega)$, and $n=1,2,\dots,N$, where 
\begin{align}\label{trun-err}
R^{n-\frac{1}{2}}_{tr}= 
{\bf i}(D_\tau u^n-u_t^{n-\frac{1}{2}})
+\Delta(\overline u^{n-\frac{1}{2}}-u^{n-\frac{1}{2}}) 
-(\overline\phi^{n-\frac{1}{2}}\overline u^{n-\frac{1}{2}}
-\phi^{n-\frac{1}{2}}u^{n-\frac{1}{2}})
\end{align}
denotes the truncation error. By Taylor expansion and the assumption \eqref{solution}, we obtain that
\begin{align}
\tau\sum^N_{n=1}\|R^{n-\frac{1}{2}}_{tr}\|^2_{L^2}\le C\tau^4.\label{truncation}
\end{align}
Subtracting \eqref{weak1}-\eqref{weak2} from \eqref{CN1}-\eqref{CN2}, we can get
\begin{align}
{\bf i}(D_\tau e_u^n,v_h) 
-(\nabla\overline e_u^{n-\frac{1}{2}},\nabla v_h)
&-(\overline\phi^{n-\frac{1}{2}}\overline u^{n-\frac{1}{2}}-\overline\phi_h^{n-\frac{1}{2}}\overline u_h^{n-\frac{1}{2}},v_h)\nonumber\\
&\qquad
=(R^{n-\frac{1}{2}}_{tr},v_h)-{\bf i}(D_\tau(u^n-R_h u^n),v_h),\label{error1}\\
\alpha(e_\phi^n,w_h)+\beta^2(\nabla e_\phi^n,\nabla w_h)
&=(|u^n|^2-|u_h^n|^2,w_h)-\alpha(\phi^n-R_h\phi^n,w_h) \label{error2}
\end{align}
for any $v_h\in \widetilde{V}_h$, $w_h\in V_h$.\\
Taking $v_h=\overline e_u^{n-\frac{1}{2}}$ in \eqref{error1}, we obtain 
\begin{align}
&
\frac{\bf i}{2\tau}(\|e_u^n\|_{L^2}^2-\|e_u^{n-1}\|_{L^2}^2)
-\|\nabla\overline e_u^{n-\frac{1}{2}}\|_{L^2}^2\nonumber \\
&
=(R^{n-\frac{1}{2}}_{tr},\overline e_u^{n-\frac{1}{2}})-{\bf i}(D_\tau(u^n-R_h u^n),\overline e_u^{n-\frac{1}{2}})+(\overline\phi^{n-\frac{1}{2}}\overline u^{n-\frac{1}{2}}-\overline\phi _h^{n-\frac{1}{2}}\overline u_h^{n-\frac{1}{2}},\overline e_u^{n-\frac{1}{2}}).\nonumber
\end{align}
Taking the imaginary parts of the above equation leads to 
\begin{align}\label{u1}
&
\frac{1}{2\tau}(\|e_u^n\|_{L^2}^2-\|e_u^{n-1}\|_{L^2}^2)\\
&\le
\left|
{\rm Im}(R^{n-\frac{1}{2}}_{tr},\overline e_u^{n-\frac{1}{2}})
-{\rm Re}(D_\tau(u^n-R_h u^n),\overline e_u^{n-\frac{1}{2}})
+{\rm Im}(\overline\phi^{n-\frac{1}{2}}\overline u^{n-\frac{1}{2}}-\overline\phi _h^{n-\frac{1}{2}}\overline u_h^{n-\frac{1}{2}},\overline e_u^{n-\frac{1}{2}})
\right|\nonumber\\
&\le 
C\|R^{n-\frac{1}{2}}_{tr}\|_{L^2}^2+C\|\overline e_u^{n-\frac{1}{2}}\|_{L^2}^2+C\|D_\tau(u^n-R_h u^n)\|_{L^2}^2+\left|{\rm Im}(\overline\phi^{n-\frac{1}{2}}\overline u^{n-\frac{1}{2}}- \overline\phi_h^{n-\frac{1}{2}}\overline u_h^{n-\frac{1}{2}},\overline e_u^{n-\frac{1}{2}})\right|\nonumber\\
&\le
C\| R^{n-\frac{1}{2}}_{tr}\|_{L^2}^2+C(\|e_u^n\|_{L^2}^2+\|e_u^{n-1}\|_{L^2}^2)+Ch^{2(r+1)}+\left|{\rm Im}(\overline\phi^{n-\frac{1}{2}}\overline u^{n-\frac{1}{2}}- \overline\phi_h^{n-\frac{1}{2}}\overline u_h^{n-\frac{1}{2}},\overline e_u^{n-\frac{1}{2}})\right|,\nonumber
\end{align}
where we use \eqref{ritz1} and \eqref{solution}  to get the last inequality. It remains to estimate the last term
\begin{align}
&
\left|{\rm Im}(\overline\phi^{n-\frac{1}{2}}\overline u^{n-\frac{1}{2}}-\overline\phi _h^{n-\frac{1}{2}}\overline u_h^{n-\frac{1}{2}},\overline e_u^{n-\frac{1}{2}})\right|\nonumber\\
&\le
\frac{1}{4}\left|{\rm Im}(\phi^n u^n-\phi_h^n u_h^n,\overline e_u^{n-\frac{1}{2}})\right|+\frac{1}{4}\left|{\rm Im}(\phi^n u^{n-1}-\phi_h^n u_h^{n-1},\overline e_u^{n-\frac{1}{2}})\right|\nonumber\\
&
\quad+\frac{1}{4}\left|{\rm Im}(\phi^{n-1}u^n-\phi_h^{n-1} u_h^n,\overline e_u^{n-\frac{1}{2}})\right|
+\frac{1}{4}\left|{\rm Im}(\phi^{n-1}u^{n-1}-\phi_h^{n-1} u_h^{n-1},\overline e_u^{n-\frac{1}{2}})\right|\nonumber\\
&=:
\sum\limits_{i=1}^4 I_i.\nonumber 
\end{align}
Based on \eqref{ritz1}-\eqref{ritz2} and \eqref{solution}, we have 
\begin{align}
I_1
&\le
\frac14\Big|{\rm Im}
((\phi^n-R_h\phi^n)u^n,\overline e_u^{n-\frac{1}{2}})+{\rm Im}(R_h\phi^n(u^n-R_h u^n),\overline e_u^{n-\frac{1}{2}})\nonumber \\ 
&
\quad+{\rm Im}(R_h\phi^n e_u^n,\overline e_u^{n-\frac{1}{2}})+{\rm Im}(R_h u^ne_\phi^n,\overline e_u^{n-\frac{1}{2}})-{\rm Im}(e_\phi^n e_u^n,\overline e_u^{n-\frac{1}{2}}) \Big|\nonumber\\
&\le 
C\|u^n\|_{L^\infty}\|\phi^n-R_h\phi^n\|_{L^2}\|\overline e_u^{n-\frac{1}{2}}\|_{L^2}
+C\|R_h\phi^n\|_{L^\infty}\|u^n-R_h u^n\|_{L^2}\|\overline e_u^{n-\frac{1}{2}}\|_{L^2}
\nonumber\\
&
\quad+C\|R_h\phi^n\|_{L^\infty}\|e_u^n\|_{L^2}\|\overline e_u^{n-\frac{1}{2}}\|_{L^2}
+C\|R_h u^n\|_{L^\infty}\|e_\phi^n\|_{L^2}\|\overline e_u^{n-\frac{1}{2}}\|_{L^2}\nonumber\\
&\quad
+C\|e_\phi^n\|_{L^6}\|e_u^n\|_{L^3}\|\overline e_u^{n-\frac{1}{2}}\|_{L^2}\nonumber\\
&\le
  C\|\phi^n-R_h\phi^n\|^2_{L^2}
  +C\|u^n-R_h u^n\|^2_{L^2}
  +C\|e^n_u\|^2_{L^2}
  +C\|e^n_{\phi}\|^2_{L^2}\nonumber\\
&\quad
  +C\|\overline e_u^{n-\frac{1}{2}}\|^2_{L^2}
  +C\|\nabla e_\phi^n\|_{L^2}\|\nabla e _u^n\|_{L^2}
  \|\overline e_u^{n-\frac{1}{2}}\|_{L^2}
  \nonumber\\
& \le
Ch^{2(r+1)}+C\|\overline e_u^{n-\frac{1}{2}}\|^2_{L^2} 
+C\|e_u^n\|^2_{L^2}+C\|e_\phi^n\|^2_{L^2}
+\varepsilon\beta^2 \|\nabla e_\phi^n\|_{L^2}^2
. \qquad\mbox{(use \eqref{ME4})}  \label{I1} 
\end{align}
Similarly, we have
\begin{align}
&I_2+I_3+I_4 \nonumber\\
&\le Ch^{2(r+1)}+C(\|e_u^n\|_{L^2}^2 +\|e_u^{n-1}\|_{L^2}^2)\nonumber\\
&\quad
+C(\|e_\phi ^n\|_{L^2}^2+\|e_\phi ^{n-1}\|_{L^2}^2)
+\varepsilon\beta^2( \|\nabla e_\phi^n\|_{L^2}^2+ \|\nabla e_\phi^{n-1}\|_{L^2}^2)
.\label{I2}
\end{align} 
Thus, substituting \eqref{I1}-\eqref{I2} into \eqref{u1}, we obtain
\begin{align}
&\frac{1}{2\tau}(\|e_u^n\|_{L^2}^2-\|e_u^{n-1}\|_{L^2}^2) \nonumber\\
&\le 
Ch^{2(r+1)}+C\|R^{n-\frac12}_{tr}\|^2_{L^2}+C(\|e_u^n\|_{L^2}^2+\|e_u^{n-1}\|_{L^2}^2)\nonumber\\
&\quad
+C(\|e_\phi ^n\|_{L^2}^2+\|e_\phi ^{n-1}\|_{L^2}^2)+
\varepsilon\beta^2( \|\nabla e_\phi^n\|_{L^2}^2+ \|\nabla e_\phi^{n-1}\|_{L^2}^2).\label{u2}
\end{align}
In order to estimate $\|e_\phi^n\|_{L^2}$, we choose $w_h=e_\phi^n$ in \eqref{error2} to get
\begin{align}
\alpha\|e_\phi^n\|_{L^2}^2
+\beta^2\|\nabla e_\phi^n\|_{L^2}^2
&\le
\left|(|u^n|^2-|u_h^n|^2,e_\phi^n)\right|
+\alpha\left|(\phi^n-R_h\phi^n,e_\phi^n)\right|\nonumber \\
&\le
\left|(|u^n|^2-|u_h^n|^2,e_\phi^n)\right| 
+\alpha\|\phi^n-R_h\phi^n\|_{L^2}\|e_\phi^n\|_{L^2} \nonumber \\
&
\le\left|(|u^n|^2-|u_h^n|^2,e_\phi^n)\right|
+Ch^{2(r+1)}+\varepsilon\alpha\|e_\phi^n\|_{L^2}^2, \label{phi1}
\end{align}
where the last inequality is obtained by \eqref{solution} and \eqref{ritz1}. By \eqref{ritz1}-\eqref{ritz2} and \eqref{solution}, we get 
\begin{align}
&\big|(|u^n|^2-|u_h^n|^2,e_\phi^n)\big|\nonumber\\
&=\big|(u^n(u^n)^*-u_h^n(u_h^n)^*,e_\phi^n)\big|\nonumber\\
&\le\big|(u^n(u^n-R_h u^n)^*,e_\phi^n)\big|
+\big|((u^n-R_h u^n)(R_h u^n)^*, e_\phi^n)\big|\nonumber\\
&\quad
+\big|(R_h u^n(e_u^n)^*, e_\phi^n)\big|
+\big|(e_u^n(R_h u^n)^*, e_\phi^n)\big|
+\big|(e_u^n(e_u^n)^*,e_\phi^n)\big|\nonumber\\
&\le
(\|u^n\|_{L^\infty}+\|R_h u^n\|_{L^\infty})\|u^n-R_h u^n\|_{L^2}\|e_\phi^n\|_{L^2}
\nonumber\\
&\quad
+C\|R_h u^n\|_{L^\infty}\|e_u^n\|_{L^2}\|e_\phi^n\|_{L^2}
+\|e_u^n\|_{L^2}\|e_u^n\|_{L^3}\|e_\phi^n\|_{L^6}\nonumber  \\
&\le 
Ch^{2(r+1)}+\varepsilon\alpha\|e^n_{\phi}\|^2_{L^2}+C\|e^n_{u}\|^2_{L^2}+C\|e_u^n\|_{L^2}\|\nabla e_u^n\|_{L^2}\|\nabla e_\phi^n\|_{L^2}\nonumber\\
&\le 
Ch^{2(r+1)}+\varepsilon
\alpha\|e_\phi^n\|_{L^2}^2
+C\|e_u^n\|_{L^2}^2+\varepsilon \beta^2\|\nabla e_\phi^n\|_{L^2}^2
.\quad\mbox{(use \eqref{ME4})} \nonumber
\end{align} 
The above inequality with \eqref{phi1} implies
\begin{align}
\alpha\|e_\phi^n\|_{L^2}^2+\beta^2\|\nabla e_\phi^n\|_{L^2}^2 
\le Ch^{2(r+1)}+C\|e_u^n\|_{L^2}^2.\label{phi2}
\end{align}
Combining \eqref{u2} and \eqref{phi2}, we obtain
\begin{align}
\frac{1}{2\tau}(\|e^n_u\|^2_{L^2}-\|e^{n-1}_u\|^2_{L^2})
\le Ch^{2(r+1)}+C\|R^{n-\frac12}_{tr}\|^2_{L^2}+C(\|e^n_u\|^2_{L^2}+\|e^{n-1}_u\|^2_{L^2}).\nonumber
\end{align}
Summing up the above result from $n=1$ to $n=m$, we get
\begin{align}
\|e_u^m\|_{L^2}^2 
\le& 
Ch^{2(r+1)}+\|e_u^0\|_{L^2}^2+C\tau\sum^m_{n=1}\|R^{n-\frac12}_{tr}\|^2_{L^2}+C\tau \sum\limits_{n=0}^m\| e_u^n\|_{L^2}^2\nonumber\\
\le&
 C\tau^4+Ch^{2(r+1)}+C\tau \sum\limits_{n=0}^m\| e_u^n\|_{L^2}^2\nonumber
\end{align}
for $m=1,2,...,N,$ where we use the fact $e^0_u=0$ and \eqref{truncation} to get the last inequality. By the discrete Gronwall's inequality, there exists a positive constant $\tau_1>0$, such that
\begin{align}
\max_{1\le n\le N}\|e_u^n\|_{L^2}^2\le C(h^{2(r+1)}+\tau^4)\nonumber
\end{align}
as $\tau<\tau_1$. Combining the above result and \eqref{phi2}, we obtain
\begin{align}
\max_{1\le n\le N}\|e_\phi^n\|_{L^2}^2 \le C(h^{2(r+1)}+\tau^4).\nonumber
\end{align}
Thus, based on the above two estimates and \eqref{ritz1}, we have 
\begin{align}
&
\|u^n-u_h^n\|_{L^2}
\le \|u^n-R_h u^n\|_{L^2}+\|e_u^n\|_{L^2}
\le C(h^{r+1}+\tau^2),\label{Sch-Hel-1}\\
&
\|\phi^n-\phi_h^n\|_{L^2}
\le \|\phi^n-R_h\phi^n\|_{L^2}+\|e_\phi^n\|_{L^2}
\le C(h^{r+1}+\tau^2)\label{Sch-Hel-2} 
\end{align}
for $n=1,2,...,N$.
\subsection{Error estimates for the Schr\"{o}dinger--Poisson equations}
In this subsection, we consider the error estimates of numerical scheme \eqref{CN1}-\eqref{CN2}  for $\alpha=0$, $\beta=1$. Taking $\alpha=0$ and $\beta=1$ in \eqref{error1}-\eqref{error2}, the error equations can be presented as follows
\begin{align} 
&
{\bf i}(D_\tau e_u^n,v_h) 
-(\nabla\overline e_u^{n-\frac{1}{2}},\nabla v_h) 
-(\overline\phi^{n-\frac{1}{2}}\overline u^{n-\frac{1}{2}}-\overline\phi _h^{n-\frac{1}{2}}\overline u_h^{n-\frac{1}{2}},v_h) 
\nonumber\\
&
\qquad=(R^{n-\frac{1}{2}}_{tr},v_h) 
-{\bf i}(D_\tau(u^n-R_h u^n),v_h),\label{error3}\\
&
(\nabla e_\phi^n,\nabla w_h)=(|u^n|^2-|u_h^n|^2,w_h)\label{error4}
\end{align}
for any $v_h\in \widetilde{V}_h$ and $w_h\in V_h$, where $R^{n-\frac{1}{2}}_{tr}$ is truncation error, defined identically as \eqref{trun-err}.\\
Taking $w_h=e_\phi^n$ in \eqref{error4} and using \eqref{ritz1}, we get
\begin{align}
\|\nabla e_\phi^n\|_{L^2}^2&=(|u^n|^2-|u_h^n|^2,e_\phi^n)=(u^n(u^n)^*-u_h^n(u_h^n)^*,e_\phi^n)\nonumber \\
&=
\big|(u^n(u^n-R_h u^n)^*, e_\phi^n)\big|
+\big|((u^n-R_h u^n)(u_h^n)^*,e_\phi^n)\big|\nonumber\\
&\qquad
+\big|(u^n(e_u^n)^*, e_\phi^n)\big|
+\big|(e_u^n(u_h^n)^*,e_\phi^n)\big|
\nonumber\\
&\le 
\|u^n\|_{L^6}\|u^n-R_h u^n\|_{L^2}\|e_\phi^n\|_{L^3}
+\|u^n-R_h u^n\|_{L^2}\|u_h^n\|_{L^6}\|e_\phi^n\|_{L^3}\nonumber\\
&\qquad+
\|u^n\|_{L^6}\|e_u^n\|_{L^2}\|e_\phi^n\|_{L^3}
+\|e_u^n\|_{L^2}\|u_h^n\|_{L^6}\|e_\phi^n\|_{L^3}
\nonumber\\
&\le 
Ch^{r+1}\|u^n\|_{L^6}\|u^n\|_{H^{r+1}}\|\nabla e_\phi^n\|_{L^2}
+Ch^{r+1}\|u^n\|_{H^{r+1}}\|\nabla u_h^n\|_{L^2}\|\nabla e_\phi^n\|_{L^2}
\nonumber\\
&\qquad
+
C\|u^n\|_{L^6}\|e_u^n\|_{L^2}\|\nabla e_\phi^n\|_{L^2}
+C\|e_u^n\|_{L^2}\|\nabla u_h^n\|_{L^2}\|\nabla e_\phi^n\|_{L^2}
\nonumber\\
&\le
Ch^{r+1}\|\nabla e_\phi^n\|_{L^2}
+C\|e_u^n\|_{L^2}\|\nabla e_\phi^n\|_{L^2},\quad
\qquad \qquad\quad\,\, \mbox{(use \eqref{solution} and \eqref{ME4})}\nonumber 
\end{align}
which implies
\begin{align}
\|\nabla e_\phi^n\|_{L^2}
\le Ch^{r+1}+C\|e_u^n\|_{L^2}.\label{phi3}
\end{align}
In order to estimate $\|e_\phi^n\|_{L^2}$, we introduce the following dual problem
\begin{align}
-\Delta \psi
=&e_\phi ^n,
\quad\,\,\,\text{in }\Omega, \label{ell-pro-1}\\
\psi=&0, 
\qquad\text{on }\partial \Omega.\label{ell-pro-2}
\end{align}
By the classic estimates of elliptic equation \cite{Evans}, it holds that 
\begin{align}
\|\psi\|_{H^2} 
\le\|e_\phi^n\|_{L^2},\label{dual}
\end{align}
for $\psi\in H^2(\Omega)$.\\
Integrating \eqref{ell-pro-1} against $e^n_\phi$ and using \eqref{error4}, we get 
\begin{align}
\|e_\phi^n\|_{L^2}^2 
&=
(\nabla\psi,\nabla e_\phi^n) 
=(\nabla\Pi_h\psi,\nabla e_\phi^n)+(\nabla(\psi-\Pi_h\psi),\nabla e_\phi^n) \nonumber \\
&=
(|u^n|^2-|u_h^n|^2,\Pi _h\psi)+(\nabla(\psi-\Pi_h\psi),\nabla e_\phi^n) \nonumber\\
&\le
\big|(u^n(u^n-R_h u^n)^*, \Pi_h\psi)\big|
+\big|((u^n-R_h u^n)(u_h^n)^*, \Pi_h\psi)\big|\nonumber\\
&\quad
+\big|(u^n(e_u^n)^*+e_u^n(u_h^n)^*, \Pi_h\psi)\big|
+\|\nabla(\psi-\Pi_h\psi )\|_{L^2}\|\nabla e_\phi^n\|_{L^2}\nonumber\\
&\le 
\|u^n\|_{L^6}\|u^n-R_h u^n\|_{L^2}\|\Pi_h\psi\|_{L^3}
+\|u^n-R_h u^n\|_{L^2}\|u^n_h\|_{L^6}\|\Pi_h\psi\|_{L^3}\nonumber\\
&\quad
+\|u^n\|_{L^6}\|e_u^n\|_{L^2}\|\Pi_h\psi\|_{L^3}
+\|e_u^n\|_{L^2}\|u_h^n\|_{L^6}\|\Pi_h\psi\|_{L^3}
+Ch\|\psi\|_{H^2}\|\nabla e_\phi^n\|_{L^2}\nonumber\\
&\le 
Ch^{r+1}\|u^n\|_{L^6}\|u^n\|_{H^{r+1}}\|\nabla\Pi_h\psi\|_{L^2}
+Ch^{r+1}\|u^n\|_{H^{r+1}}\|\nabla u^n_h\|_{L^2}\|\nabla \Pi_h\psi\|_{L^2}\nonumber\\
&\quad
+C\|u^n\|_{L^6}\|e_u^n\|_{L^2}\|\nabla \Pi_h\psi\|_{L^2}
+C\|e_u^n\|_{L^2}\|\nabla u_h^n\|_{L^2}\|\nabla \Pi_h\psi\|_{L^2}
+Ch\|\psi\|_{H^2}\|\nabla e_\phi^n\|_{L^2}\nonumber\\
&\le
C(h^{r+1}+\|e_u^n\|_{L^2})\|e^n_{\phi}\|_{L^2}+Ch\|e_\phi^n\|_{L^2}\|\nabla e_\phi^n\|_{L^2},    \label{phi4}
\end{align}
where we use \eqref{ME4}, \eqref{interpolation} and \eqref{dual} to get the last inequality. From \eqref{phi3} and \eqref{phi4}, we can get
\begin{align}
\|e_\phi^n\|_{L^2}\le Ch^{r+1}+C\|e_u^n\|_{L^2}.\label{phi5}
\end{align}
Since the error equation \eqref{error3} is equivalent to \eqref{error1}, we substitute $v_h=\overline{e}^{n-\frac12}_u$ into \eqref{error3}, and apply the same analysis as \eqref{u1}-\eqref{u2} to obtain the following estimate:
\begin{align}
&\frac{1}{2\tau}(\|e_u^n\|_{L^2}^2-\|e_u^{n-1}\|_{L^2}^2) \nonumber\\
&\le 
Ch^{2(r+1)}+C\|R^{n-\frac12}_{tr}\|^2_{L^2}+C(\|e_u^n\|_{L^2}^2+\|e_u^{n-1}\|_{L^2}^2)\nonumber\\
&\quad
+C(\|e_\phi ^n\|_{L^2}^2+\|e_\phi ^{n-1}\|_{L^2}^2)+
\varepsilon\beta^2( \|\nabla e_\phi^n\|_{L^2}^2+ \|\nabla e_\phi^{n-1}\|_{L^2}^2).\label{u4}
\end{align}
Substituting \eqref{phi3}, \eqref{phi5} into \eqref{u4}, and summing up from $n=1$ to $n=m$, we get
\begin{align*}
\|e_u^m\|_{L^2}^2 
\le& Ch^{2(r+1)}+\|e_u^0\|_{L^2}^2+C\tau\sum^m_{n=1}\|R^{n-\frac12}_{tr}\|^2_{L^2}+C\tau\sum\limits_{n=0}^m\|e_u^n\|_{L^2}^2\\
\le&
C(h^{2(r+1)}+\tau^4)+C\tau\sum\limits_{n=0}^m\|e_u^n\|_{L^2}^2
\end{align*}
for $m=1,2,...,N,$ where we use \eqref{truncation} and $e^0_u=0$. By the discrete Gronwall's inequality, there exists a positive constant $\tau_2>0$, such that
\begin{align}
\max_{1\le n\le N}\|e_u^n\|_{L^2}^2\le C(h^{2(r+1)}+\tau^4)\nonumber
\end{align}
as $\tau<\tau_2$.\\
Combining the above result and \eqref{phi5}, we have
\begin{align}
\max_{1\le n\le N}\|e_\phi^n\|_{L^2}^2 \le C(h^{2(r+1)}+\tau^4).\nonumber
\end{align}
By the estimate \eqref{ritz1} of Ritz projection, we have 
\begin{align}
&
\|u^n-u_h^n\|_{L^2}\le \|u^n-R_h u^n\|_{L^2} +\|e_u^n\|_{L^2}
\le C(h^{r+1}+\tau^2),\label{Sch-poss-1}\\
&
\|\phi^n-\phi_h^n\|_{L^2}\le \|\phi^n-R_h\phi^n\|_{L^2}+\|e_\phi^n\|_{L^2}
\le C(h^{r+1}+\tau^2)\label{Sch-poss-2}
\end{align}
for $n=1,2,...,N$. 
Thus, based on Theorem \ref{TH4-2}, \eqref{Sch-Hel-1}-\eqref{Sch-Hel-2} and \eqref{Sch-poss-1}-\eqref{Sch-poss-2}, by setting $\tau_0=\min\{\tau^*, \tau_1, \tau_2\}$, we complete the proof of Theorem \ref{main}.

\section{Numerical Examples}
In this section, we present some numerical examples to verify the convergence rate and conservation properties of our scheme \eqref{CN1}-\eqref{CN2}. All the computations are performed by FreeFem++.

\begin{example}\normalfont \label{example} 
We consider the following two-dimensional Schr\"{o}dinger--Helmholz system
\begin{align}
&\quad\qquad\qquad\qquad\qquad\qquad{\bf i}u_t+\Delta u-\phi u=f_1,
&&{\bf x}\in\Omega,\, 0<t\le T ,\label{ex-equ-1}\\
&\quad\qquad\qquad\qquad\qquad\qquad\phi-\Delta\phi=|u|^2+f_2,
&&{\bf x}\in\Omega, \, 0<t\le T,\label{ex-equ-2}\\
&\quad\qquad\qquad\qquad\qquad\qquad u({\bf x},0)=u_0({\bf x}),\quad\phi({\bf x},0)=\phi_0({\bf x}),
&& {\bf x}\in\Omega,\label{ex-equ-3}\\
&\quad\qquad\qquad\qquad\qquad\qquad u({\bf x},t)=\phi({\bf x},t)=0,
&&{\bf x}\in\partial\Omega, \, 0<t\le T,\label{ex-equ-4}
\end{align}
where $\Omega=[0,1]\times[0,1] $ and we take $T=0.5$. Moreover, the initial conditions $u_0$, $\phi_0$ and right-hand side function $f_1$, $f_2$ are determined by the following exact solutions
\begin{align}
&u=e^{({\bf i}+1)t}\sin(x)\sin(y)\sin(\pi x)\sin(\pi y),\nonumber\\
&\phi=e^{(t+x+y)}(1-x)(1-y)\sin(x)\sin(y).\nonumber
\end{align}
The system \eqref{ex-equ-1}-\eqref{ex-equ-4} is discretized by the implicit Crank--Nicolson FE method \eqref{CN1}-\eqref{CN2} with linear and quadratic FE approximation. Since our algorithm \eqref{CN1}-\eqref{CN2} is nonlinear, the following Picard's iteration is used to obtain the numerical solution $(u^n_h, \phi^n_h)$ at each time step.\vspace{0.08cm}\\
{\bf Picard's iteration for \eqref{CN1}-\eqref{CN2} }\vspace{0.05cm}\\
{\bf Step 1.}
Initialization for the time marching: Let the time step $n=0$, and set the initial value $u_h^0=R_h u_0$, $\phi_h^0=R_h\phi_0$.\\
{\bf Step 2.}
Initialization for nonlinear iteration: For $n\ge1$, set $u_h^{n,0}=u_h^{n-1}$ and define by $\phi_h^{n,0}$ the solution of the following system:
\begin{align}
(\phi_h^{n,0},w_h)+(\nabla\phi_h^{n,0},\nabla w_h)=(|u_h^{n,0}|^2,w_h)+(f_2^n,w_h),\qquad\forall\, w_h\in V_h.\nonumber
\end{align}
{\bf Step 3.}
FE computation on each time level: For $\ell\ge 0$, find $u_h^{n,\ell+1}\in\widetilde{V}_h$, such that
\begin{align}
&{\bf i}\bigg( \frac{u_h^{n,\ell+1}-u_h^{n-1}}{\tau},v_h \bigg) 
- \bigg( \nabla \frac{u_h^{n,\ell+1} + u_h^{n-1}}{2},\nabla v_h\bigg) 
- \bigg(\frac{\phi_h^{n,\ell}+\phi_h^{n-1}}{2}\frac{u_h^{n,\ell+1}+u_h^{n-1}}{2},v_h\bigg)\nonumber\\ 
&= \bigg(f_1^{n-\frac{1}{2}},v_h\bigg),\qquad \forall\, v_h\in \widetilde{V}_h,\nonumber
\end{align}
and find $\phi_h^{n,\ell+1}\in V_h$, such that 
\begin{align}
(\phi_h^{n,\ell+1},w_h)+(\nabla\phi_h^{n,\ell+1},\nabla w_h)=(|u_h^{n,\ell+1}|^2,w_h)+(f_2^n,w_h),\qquad\forall\, w_h\in V_h.\nonumber
\end{align}
\\[-0.35cm]
{\bf Step 4.}
To check the stopping criteria for nonlinear iteration: For a fixed tolerance $\delta=10^{-7}$, stop the iteration when
\begin{align}
\|u_h^{n,\ell+1}-u_h^{n,\ell}\|_{L^2}
\le\delta,\quad\| \phi_h^{n,\ell+1}-\phi_h^{n,\ell}\|_{L^2}\le\delta, \nonumber
\end{align} 
and set $u_h^{n}=u_h^{n,\ell+1}$,  $\phi_h^{n}=\phi_h^{n,\ell+1}$. Otherwise, set $\ell \leftarrow \ell+1$, and go to {\bf Step 3} to continue the nonlinear iteration.\\
{\bf Step 5.}
Time marching: if $n=N$, stop the iteration. Otherwise, set $n-1\leftarrow n$, and go to {\bf Step 2}.\vspace{0.08cm}

To test the convergence order, we choose $\tau=h$ for linear FE approximation and $\tau=h^{\frac32}$ for quadratic FE approximation, respectively. Obviously, numerical results in Table \ref{space1} and Table \ref{space2} are consistent with the theoretical analysis in Theorem \ref{main}. 
\begin{table}
\renewcommand\arraystretch{1.26} 
\centering 
\hspace{-2.5mm}
\caption{$L^2$-norm errors of linear FE method with $\tau=h$}
\vspace{0.2cm}
\begin{tabular}{c|c|c|c|c}
\hline
$ h$ & $\|{\bf u}^n-{\bf u}_h^n\|_{L^2}$ & order & $\|\phi^n-\phi_h^n\|_{L^2}$& order    \\
\hline
$1/10$ &7.1770E-03 &/& 3.9596E-03   &/  \\
$1/20$&1.9732E-03 &1.86& 1.0579E-03 &1.90 \\
$1/40$&5.2178E-04 &1.92 & 2.7276E-04 &1.96\\
\hline
\end{tabular}\label{space1}
\end{table}

\begin{table}\label{5.2}
\renewcommand\arraystretch{1.26}   
\centering 
\hspace{-2.5mm}
\caption{$L^2$-norm errors of quadratic FE method with $\tau=h^{\frac32}$}
\vspace{0.2cm}
\begin{tabular}{c|c|c|c|c}
\hline
$ h$ & $\|{\bf u}^n-{\bf u}_h^n\|_{L^2}$ & order & $\|\phi^n-\phi_h^n\|_{L^2}$& order    \\
\hline
$1/10$&2.0123E-04&/&8.5171E-05&/  \\
$1/20$&2.5669E-05&2.97 &1.0736E-05&2.99 \\
$1/40$&3.1864E-06&3.01 &1.3468E-06&2.99 \\
\hline
\end{tabular}\label{space2}
\end{table}
\end{example}

\begin{example}\normalfont \label{example2} 
We consider the following two-dimensional Schr\"{o}dinger--Poisson system
\begin{align}
&\quad\qquad\qquad\qquad\qquad\qquad{\bf i}{u_t}+\Delta u-\phi u=g_1,
&&{\bf x}\in\Omega, \,0<t\le T,\label{Exam2-1}\\
&\quad\qquad\qquad\qquad\qquad\qquad-\Delta\phi=|u|^2+g_2,
&&{\bf x}\in\Omega, \,0<t\le T,\label{Exam2-2}\\
&\quad\qquad\qquad\qquad\qquad\qquad u({\bf x},0)=u_0({\bf x}),\quad\phi({\bf x},0)=\phi_0({\bf x}),
&&{\bf x}\in\Omega,\label{Exam2-3}\\
&\quad\qquad\qquad\qquad\qquad\qquad u({\bf x},t)=\phi({\bf x},t)=0,
&&{\bf x}\in\partial\Omega, \, {0<t\le T},\label{Exam2-4}
\end{align}
where $\Omega=[0,1]\times[0,1]$ and we take $T=$ 0.5. Moreover, the function $u_0$, $\phi_0$ and $g_1$, $g_2$ are determined by the following exact solutions
\begin{align*}
&u=2e^{{\bf i}t+(x+y)/5}(1+5t^3)x(1-x)y(1-y),\\ 
&\phi=5(1+3t^2+\sin(t))\sin\left(\frac{x}{2}\right)\sin\left(\frac{x}{2}\right)(1-x)(1-y).
\end{align*}

We solve the system \eqref{Exam2-1}-\eqref{Exam2-4} by Crank--Nicolson FE scheme \eqref{CN1}-\eqref{CN2} with both linear and quadratic approximation. Similar to the Example \ref{example}, we apply Picard's iteration method to solve our nonlinear scheme \eqref{CN1}-\eqref{CN2} for $\alpha=0$ and $\beta=1$. To verify the convergence rate, we choose $\tau=h$ and $\tau=h^{\frac32}$ for linear and quadratic FE approximation, respectively. The numerical results are presented in Table \ref{space3} and Table \ref{space4}. These results indicate that the $L^2$ error results of linear FE approximation are proportional to $h^2$, and the $L^2$ error results of quadratic FE approximation are proportional to $h^3$, which are consistent with theoretical analysis.
\begin{table}
\renewcommand\arraystretch{1.26}  
\centering 
\hspace{-2.5mm}
\caption {$L^2$-norm errors of linear FE method with $\tau=h$ }
\vspace{0.2cm}
\begin{tabular}{c|c|c|c|c}
\hline
$ h$ & $\|{\bf u}^n-{\bf u}_h^n\|_{L^2}$ & order & $\|\phi^n-\phi_h^n\|_{L^2}$& order  \\ 
\hline					
$1/10$&3.4536E-03 &/& 2.1514E-03&/  \\
$1/20$&1.1618E-03 &1.57&5.9753E-04&1.85\\
$1/40$&3.0870E-04 &1.91&1.5686E-04&1.93 \\
\hline
\end{tabular}\label{space3}
\end{table}
\begin{table}
\renewcommand\arraystretch{1.26}  
\centering 
\hspace{-2.5mm}
\caption{$L^2$-norm errors of quadratic FE method with $\tau=h^{\frac32}$}
\vspace{0.2cm}
\begin{tabular}{c|c|c|c|c}
\hline
$h$ & $\|{\bf u}^n-{\bf u}_h^n\|_{L^2}$ & order & $\|\phi^n-\phi_h^n\|_{L^2}$& order    \\
\hline 			 
$1/10$&1.5718E-04&/&3.7893E-05&/  \\
$1/20$&2.0529E-05&2.94 &4.8388E-06&2.97 \\
$1/40$&2.5437E-06&3.01 &6.0897E-07&2.99 \\
\hline
\end{tabular}\label{space4}
\end{table}
\end{example}

\begin{example}\normalfont\label{example3} 
To verify the discrete mass and energy conservation, we consider the following Schr\"{o}dinger--type system
\begin{align}
&\qquad\qquad\qquad\qquad\qquad{\bf i}u_t+\Delta u-\phi u=0,
&&{\bf x}\in\Omega,\, 0<t\le T,\label{eg31}\\   
&\qquad\qquad\qquad\qquad\qquad\alpha\phi-\beta^2\Delta\phi=|u|^2,
&&{\bf x}\in \Omega,\, 0<t\le T,\label{eg32}\\ 
&\qquad\qquad\qquad\qquad\qquad u({\bf x},0)=u_0({\bf x}),\quad\phi({\bf x},0)=\phi_0({\bf x}),
&&{\bf x}\in\Omega,\label{eg33}\\  
&\qquad\qquad\qquad\qquad\qquad u({\bf x},t)=\phi({\bf x},t)=0 ,
&&{\bf x}\in\partial\Omega ,\, {0<t\le T}\label{eg34}
\end{align}
with the initial conditions
\begin{align}
&u({\bf x},0)=\sin(x)\sin(y)\sin(\pi x)\sin(\pi y), \nonumber\\
&\phi({\bf x},0)=e^{x+y}(1-x)(1-y)\sin(x)\sin(y),\nonumber
\end{align}
where $\Omega=[0, 1]\times[0, 1]$.
To test mass and energy conservation, we apply the Crank--Nicolson linear FE method to solve the above system \eqref{eg31}-\eqref{eg34}, which includes two models:

\vspace{-5.8pt}

\begin{itemize}
\setlength{\itemsep}{0pt}
\setlength{\parsep}{0pt}
\item Schr\"{o}dinger--Helmholz system (as $\alpha=1$ and $\beta=1$) 
\item Schr\"{o}dinger--Poisson system (as $\alpha=0$ and $\beta=1$)
\end{itemize}

\vspace{-5.8pt}

The numerical results at different time stages $T=0,10,20,...,100$ are presented in Fig.\ref{figure}.  From Fig.\ref{figure}, we can see that the discrete energy $\mathcal{E}^n_h$ and mass $\mathcal{M}^n_h$ are conserved exactly during time evolution, which consistent with the theoretical analysis in Theorem \ref{energy}.

\begin{figure}[htp]
\centering
\subfloat[Schr\"{o}dinger--Helmholz system]{
\epsfig{file=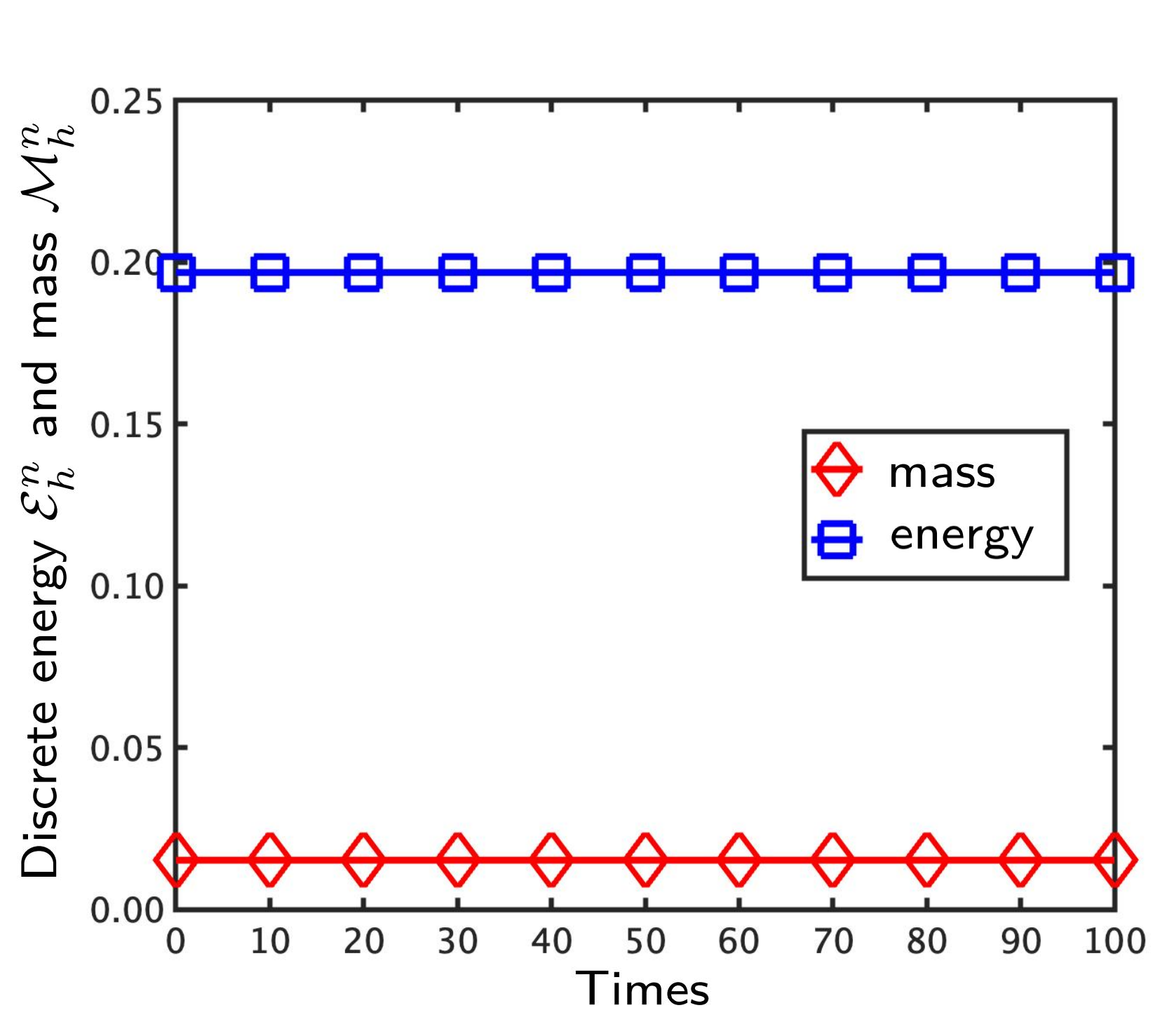,height=2in,width=2.13in}}
\hspace{6mm}
\subfloat[Schr\"{o}dinger--Poisson system]{
\epsfig{file=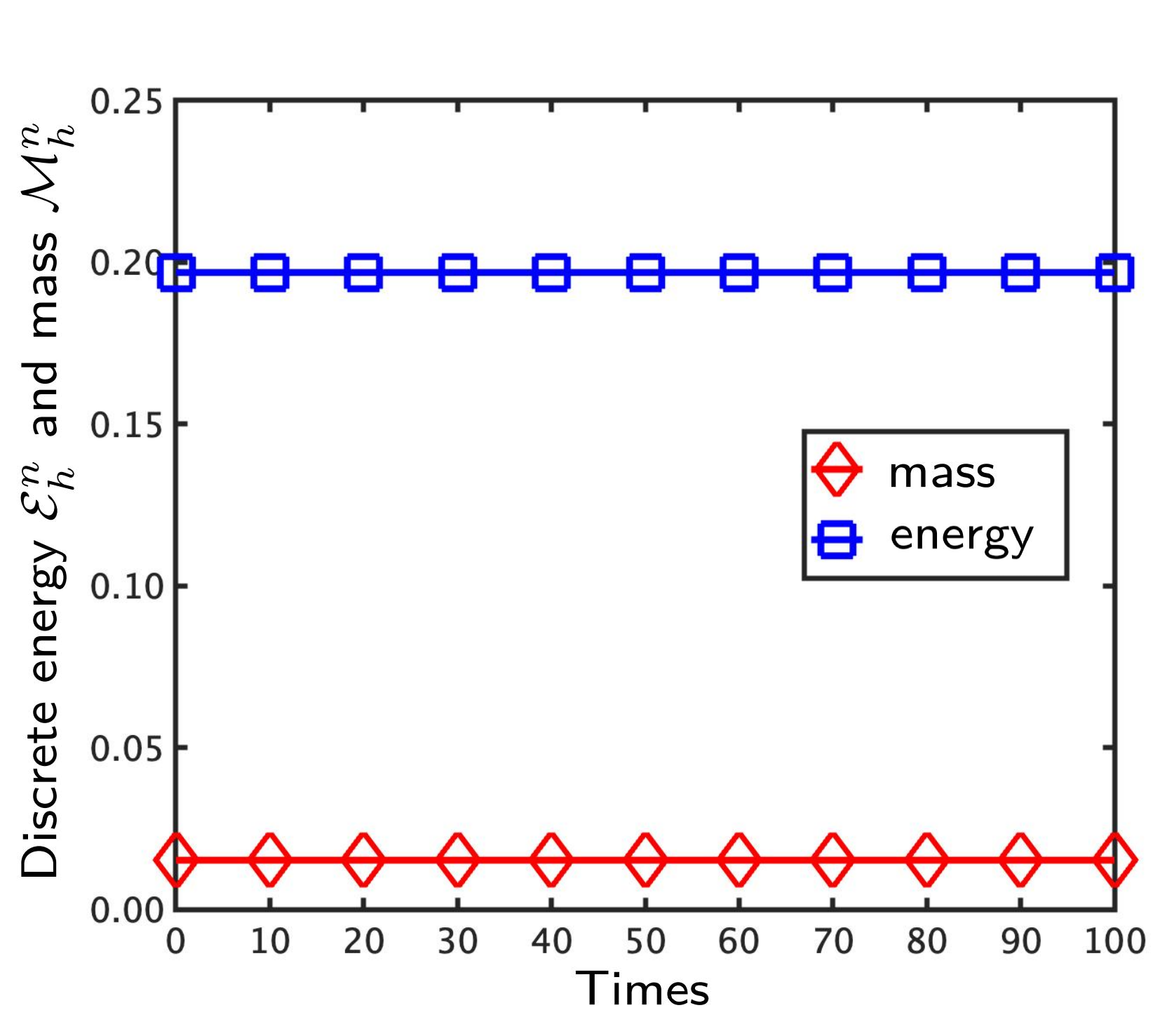,height=2in,width=2.13in}}
\caption{Evolution of discrete energy $\mathcal{E}_h^n $ and mass $\mathcal{M}_h^n$ }\label{figure}
\end{figure}
\end{example}
\section{Conclusions}
In this paper, an implicit Crank--Nicolson FE scheme is present to solve a nonlinear Schr\"{o}dinger--type system, which includes Schr\"{o}dinger--Helmholz equations and Schr\"{o}dinger--Poisson equations. The advantage of our numerical method is mass and energy conservation in the discrete level. The well-posedness (existence and uniqueness) of the fully discrete solutions are proved by Schaefer's fixed point theorem. We demonstrate optimal $L^2$ error estimates for the fully discrete solutions for Schr\"{o}dinger--Helmholz equations and Schr\"{o}dinger--Poisson equations. Finally, some numerical examples are provided to confirm the theoretical analysis.

\section{Acknowledgments}

The work of Zhuoyue Zhang and Wentao Cai was partially supported by
the Zhejiang Provincial Natural Science Foundation of China grant LY22A010019, the National Natural Science Foundation of China grant 11901142 and Fundamental Research Funds for the Provincial Universities of Zhejiang grant GK219909299001-025.

\end{document}